\documentclass[11pt,reqno]{amsart}

\usepackage{amsthm,amsmath,amsfonts,amssymb,bm,bbm,mathtools} 
\usepackage{amsaddr}
\usepackage{graphicx}

\usepackage{datetime} \usdate
\usepackage{currfile} 
\usepackage{color}
\usepackage[margin=30mm]{geometry}
\usepackage{subfigure}

\date{\today}
\newif\ifdraft

\newcommand{\red}{} 

\ifdraft
\renewcommand{\red}{\color{red}}

\fi

\iffalse 

 \usepackage[landscape,includehead,headheight=12pt,headsep=10pt,left=\columnsep,right=\columnsep,top=12.98999pt,bottom=\columnsep,]{geometry}
 \twocolumn
\else
\fi

\usepackage{pdfsync,hyperref}

\newtheorem{theorem}{Theorem}[section]

\newtheorem{corollary}[theorem]{Corollary}

\newtheorem{lemma}[theorem]{Lemma}
\newtheorem{proposition}[theorem]{Proposition}

\theoremstyle{remark}
\newtheorem{remark}[theorem]{Remark}

\newcommand{\nwc}{\newcommand}

\nwc{\pdfx}[2]{\texorpdfstring{#1}{#2}} 
\nwc{\hide}[1]{}  
\nwc{\qref}[1]{(\ref{#1})}
\nwc{\ip}[1]{\langle #1 \rangle}
\nwc{\sfrac}[2]{{\textstyle \frac{#1}{#2}}}

\nwc{\D}{\partial}
\nwc{\grad}{\nabla}
\nwc{\eps}{\varepsilon}
\nwc{\inv}{^{-1}}
\nwc{\wkto}{\xrightharpoonup{\star}}

\nwc{\sgn}{\mathop{\rm sgn}\nolimits}
\nwc{\tr}{\mathop{\rm tr}\nolimits}
\nwc{\re}{\mathop{\rm Re}\nolimits}
\nwc{\im}{\mathop{\rm Im}\nolimits}

\renewcommand{\Re}{\mathop{\rm Re}\nolimits}
\renewcommand{\Im}{\mathop{\rm Im}\nolimits}

\nwc{\one}{{\mathbbm{1}}}
\nwc{\calA}{{\mathcal A}}
\nwc{\calB}{{\mathcal B}}
\nwc{\calC}{{\mathcal C}}
\nwc{\calD}{{\mathcal D}}

\nwc{\calH}{{\mathcal H}}
\nwc{\calL}{{\mathcal L}}
\nwc{\calK}{{\mathcal K}}
\nwc{\calM}{{\mathcal M}}
\nwc{\calP}{{\mathcal P}}
\nwc{\calU}{{\mathcal U}}
\nwc{\calV}{{\mathcal V}}
\nwc{\calW}{{\mathcal W}}
\nwc{\N}{\mathbb{N}}
\nwc{\Z}{\mathbb{Z}}
\nwc{\C}{\mathbb{C}}
\nwc{\R}{\mathbb{R}}
\nwc{\bbD}{\mathbb{D}}
\nwc{\bbP}{\mathbb{P}}
\nwc{\bbT}{\mathbb{T}}
\nwc{\bbZ}{\mathbb{Z}}
\nwc{\mb}[1]{\mathbf{#1}}

\nwc{\ups}{\mu}
\nwc{\Drho}{\calD_\rho}

\nwc{\tbar}{{\bar t}}

\DeclareMathOperator{\sn}{sn}

\nwc{\esssup}{\mathop{\rm ess\, sup}}

\DeclareMathOperator{\diag}{diag}

\definecolor{mygreen}{rgb}{0.1,0.75,0.2}

\nwc{\wksto}{\xrightharpoonup{\star}}
\nwc{\tlpto}{\overset{{T\!L^p}}{\longrightarrow} }
\nwc{\tltwoto}{\overset{{T\!L^2}}{\longrightarrow} }
\nwc{\tlp}{{T\!L^p}}
\nwc{\tltwo}{{T\!L^2}}

\renewcommand{\th}{\theta}
\nwc{\bodypot}{\calV}
\nwc{\tcr}[1]{ {\red #1}}
\nwc{\pow}{\nu}
\nwc{\vel}{v}

\date{\today} 
\ifdraft \date{WORKING DRAFT \today: PLEASE DO NOT CIRCULATE} \fi

\begin{document}

\ifdraft
\title[DRAFT \mmddyyyydate\today\quad \currenttime \ \ PLEASE DO NOT CIRCULATE]
{In search of local singularities in \\
ideal potential flows with free surface}

\else
\title[In search of local singularities]
{In search of local singularities in \\
ideal potential flows with free surface}
\fi

\author{Jian-Guo Liu}
\address{
Department of Physics and Department of Mathematics\\
Duke University, Durham, NC 27708, USA}
\email{jliu@phy.duke.edu}

\author[Jian-Guo Liu and Robert L. Pego]{Robert L. Pego}
\address{Department of Mathematical Sciences and Center for Nonlinear Analysis\\
Carnegie Mellon University, Pittsburgh, Pennsylvania, PA 12513, USA}
\email{rpego@cmu.edu}

\begin{abstract}  %
For ideal fluid flow with zero surface tension and gravity,
it remains unknown whether local singularities on the free surface
can develop in well-posed initial value problems with smooth initial data.
This is so despite great advances over the last 25 years in the mathematical
analysis of the Euler equations for water waves. 
Here we expand our earlier work ({\em Chin. Ann. Math. Ser. B} 40 (2019) 925)
and review the mathematical literature and some of the history 
concerning Dirichlet's ellipsoids and 
related hyperboloids associated with jet formation and ``flip-through,'' 
``splash singularities,'' 
and recent constructions of singular free surfaces that 
however violate the Taylor sign condition for linear well-posedness.
We illustrate some of these phenomena with numerical computations 
of 2D flow based upon a conformal mapping formulation 
(whose derivation is detailed and discussed in an appendix).
Additional numerical evidence strongly suggests
that corner singularities may form in an unstable self-similar way 
from specially prepared initial data.
\end{abstract}


\maketitle
\tableofcontents


\section{Introduction} \label{s:intro}

The highly nonlinear behavior of fluids is an endless source of fascination and challenges 
to our understanding. Much mathematical analysis can only deal with models of smooth, 
rather quiescent flows.
Yet some of the most powerful and dramatic fluid phenomena are associated with singular flows.
What happens when waves crash against a seawall?
How do whitecaps form on windblown waves?  How do droplets shatter and become spray?

Our understanding of such phenomena is very primitive. We consider here 
the question of singularity formation for one of the simplest fluid models,
Euler's equations for potential flow of an ideal fluid. 
With velocity field $v=\nabla\phi$, pressure $p$, and constant density $\rho=1$,
occupying a region $\Omega_t\subset\R^d$ with smooth boundary at time $t$,
these equations take the following form: For each time $t$,
\begin{align}
\Delta\phi = 0 & 
\quad \mbox{in $\Omega_t$}, 
\label{1.laplace}\\
\phi_t + \frac12|\nabla\phi|^2 + p = 0&
\quad \mbox{in $\Omega_t$}, 
\label{1.bernoulli}\\
p = 0 & 
\quad \mbox{on $\D\Omega_t$}. 
\label{1.pzero}
\end{align}
The {\em kinematic condition}, stating that the fluid domain $\Omega_t$ is transported by the velocity,
supplements these equations.
The effects of gravity and surface tension are neglected.
At the small scales involved in singularity formation,  
it is generally appreciated that the effect of gravity ought to be negligible.
It is true that surface tension is physically important on small scales, 
but we focus this study on the mathematical issues that arise when it is neglected.

It is our purpose here to extend our previous work \cite{LPmajda}
reviewing research relevant to the issue of whether and how local singularities
can form in solutions of this system, and offer additional numerical evidence
that suggests a new scenario for formation of a local singularity.
In the sequel, particularly when considering bounded domains $\Omega_t$
we refer to \eqref{1.laplace}--\eqref{1.pzero} simply as the {\em ideal droplet equations}.

\section{Background---scenarios for singularities}

Mathematical analysis of the initial-value problem for the governing equations
\eqref{1.laplace}--\eqref{1.pzero} with free boundary is subtle and difficult.
The problem can be treated, however, by the methods that S. Wu developed in the 1990s for
water waves with gravity.  For smooth enough initial data in smooth bounded domains,
the works \cite{Lindblad,CoutShko2007,CoutShko2010}  establish short-time
existence for smooth solutions of the incompressible Euler equations with
pressureless free boundary in zero gravity, including the case of nonzero vorticity.

In this section we briefly review work related to a number of scenarios for 
the possible breakdown of smooth solutions and development of singularities in solutions.
Bounds that constrain local singularity formation have been provided in 
recent work of Kinsey and Wu~\cite{KinseyWu2018} and Wu~\cite{wu2015blow}.
In the latter work it is also shown that certain kinds of corners in the free surface
can persist for short time if they are present in the initial data.

{\em Splash singularities.}
A simple way that fluids can develop a singularity is by collision of distinct droplets.
A related but more complicated scenario is that different parts of the surface of a 
connected fluid domain may collide, while the interface remains smooth up to the time of collision.
The existence of such {\em splash singularities} 
was proved by Castro et al.~\cite{Castro2012,Castro2013}.

{\em Droplet splitting.}  
One can imagine that a single dumbbell-shaped droplet provided with a strongly bipolar 
initial velocity should break in pieces.   There are many physical studies of
this behavior that take into account surface tension and/or viscosity. We are not aware 
of any study of the problem in the absence of these effects, however, and it may be that surface tension
is necessary after all for pinch-off to occur.  One idea for approaching the splitting problem 
could involve finding a {\em least-action} path of fluid configurations that deform one droplet into two.
Smooth incompressible potential flows in a fixed domain were shown by Brenier \cite[Theorem 2.4]{Brenier99}
to truly minimize action for sufficiently short time. 
These flows correspond to volume-preserving paths of diffeomorphisms that minimize 
distance according to a relaxed version of Arnold's variational characterization of 
geodesic paths in the diffeomorphism group \cite{Arnold66}.
It was recently proved, however, that free-boundary flows with zero gravity and
surface tension are critical paths  for action, but {\em never minimize it}
except for piecewise-rigid motions~\cite[Corollary~5.6]{LPS}.

{\em Flip-through and jet formation.}  The breaking of gravity waves against a vertical wall can be thought of
as a kind of splash singularity, by reflecting the fluid motion through the plane of the wall.
In work of Cooker and Peregrine \cite{cooker1990computations,cooker1992violent}
2D numerical computations show that wave impacts that trap a bubble of `air' are less violent than 
waves that only get {\em close} to breaking at the wall.  
Strong forces and very large accelerations can be produced as a sheet of water ``flips through''
the trough and generates a powerful upward jet of fluid.
For discussion of the flip-through phenomenon and related experiments
see \cite{Peregrine2003,Bredmose2010,wang2018}.

In a series of papers including 
\cite{Longuet1972,Longuet1976,Longuet1980,Longuet1983},
Longuet-Higgins described a jet-formation phenomenon that appears to be associated with 
flip-through and some other situations where the 
na\"ive expectation is that local singularities might form.
In particular, Longuet-Higgins described ``Dirichlet hyperboloid'' exact solutions,
extending a family of time-dependent ellipsoidal solutions found by Dirichlet~\cite{dirichlet1860}
in relation to a long line of investigations on ellipsoidal self-gravitating fluid bodies.
As Longuet-Higgins mentioned, Fritz-John \cite{john1953} had found related flows with parabolic free boundary.

Longuet-Higgins compared Dirichlet hyperboloid solutions with experiments on breaking waves and bubbles
in \cite{Longuet1983}.  All the flows he observed remain smooth. 
There are Dirichlet hyperboloid solutions that become singular in finite time, but what happens is that the 
pressure and velocity blow up everywhere while the fluid interface remains smooth. 
By taking a large-scale limit, in \cite{Longuet1980} Longuet-Higgens described 
time-dependent solutions that have corners for all times.

{\em Self-similar approach to cones or corners.}
The tendency to form jets with smooth tips may make it unlikely that 
smooth free boundaries develop local singularities in many typical flows.
We might expect that a local singularity may appear in a borderline situation,
e.g., between strong flip-through and bubble-trapping splash singularity. 
Experimental and numerical evidence of such a singularity, 
for 3D incompressible flows with viscosity and surface tension,
was provided by D.~Lathrop's group in the 1990s~\cite{lathrop1998,zeff2000singularity}.
These authors demonstrated a self-similar collapse of the fluid interface to one with a conical singularity,
followed by the dramatic emergence of a very high and thin self-similar jet.
No rigorous mathematical analysis of this problem has yet appeared,
as far as we are aware.

{\em Ballistic interfaces.} In the papers 
\cite{KZjfm2014,KZdok2016}, Karabut and Zhuravleva 
described several analytical solutions of the free-boundary problem 
\eqref{1.laplace}--\eqref{1.pzero} 
for which fluid particles on the free surface move with {\em zero acceleration}, i.e., they move ballistically.
Very recently, Zubarev and Karabut \cite{ZK2018} and Zhuravleva et al.~\cite{zhuravleva2020new} 
have described examples of this type of flow capable of developing local singularities 
from a smooth interface.
These solutions are derived using the complex Hopf equation by imposing a particular relation between 
pressure and velocity at the free boundary.  
Zakharov~\cite{Zak20} has provided an interesting independent perspective on solutions of this type.

At present it seems doubtful that the singularities of the kind found in these works
can emerge in smooth flows in bounded domains.  For acceleration-free interfaces, 
both the pressure and its gradient vanish at the free boundary. In addition,
the pressure-velocity relations imposed in \cite{ZK2018,zhuravleva2020new}
imply that the pressure $p<0$ inside the fluid domain.  

However, it is necessary that $p>0$ inside the fluid for any nontrivial smooth solution of
\eqref{1.laplace}--\eqref{1.pzero} in a bounded domain.
This follows from the fact that $-\Delta p = \Delta|\nabla\phi|^2\ge0$.
Then Hopf's lemma implies that
\begin{equation}\label{e:taylor}
\frac{\D p}{\D n} <0  \quad\mbox{on $\D\Omega_t$}.
\end{equation}
In the present context this says that the {\em Taylor sign condition} for linear well-posedness holds. 
More generally this condition requires the outward normal acceleration of the
interface to exceed the acceleration due to gravity, see \cite{taylor50,bhl93}.
It was recognized to be key to nonlinear well-posedness theory by Wu \cite{Wu97,Wu99}.

{\em Plan of the paper.}
In sections~\ref{s:leastaction}--\ref{s:dirichlet} below, we aim to describe
the explicit Dirichlet ellipsoid and hyperboloid solutions of the ideal 
droplet equations \eqref{1.laplace}--\eqref{1.pzero}, with a view to focus on their significance 
for the droplet splitting and jet formation scenarios mentioned above.
These solutions exist in an historical context that is interesting to review,
involving Hamilton's least action principle, kinematically constrained geodesic flow, 
and a nontrivial symmetry exhibited by self-gravitating bodies
that was made explicit by Dekekind when preparing Dirichlet's work for posthumous publication.

In section~\ref{s:ZK}, we summarize how local singularities on ballistic interfaces 
were derived in~\cite{ZK2018} for purely horizontal surface motions
and in \cite{zhuravleva2020new} for cavity collapse scenarios.

Then in the last section below, we extend our computations from~\cite{LPmajda}
to provide additional evidence for a scenario involving unstable corner formation.
We make use of a conformal mapping formulation of the governing equations
closely related to one described by A.~I.~Dyachenko in \cite{dyachenko2001dok}
and used by S.~A.~Dyachenko in \cite{dyachenko2019jfm} to compute bounded ideal droplet solutions
with and without surface tension.
We find evidence for the existence of a two-parameter of self-similar smooth flows that may
emerge from an infinite perfect wedge-shaped domain with power-law initial velocity,
by computing a time-reversed flow that develops from a smooth bounded approximation to the wedge,
together with a scaling argument.  The problem of rigorously demonstrating the existence 
of such solutions (or showing that some other instabilities must occur on scales invisible to our numerics)
appears to pose a difficult challenge for mathematical analysis.

\section{Least action principle with free boundary and self-interaction energy}
\label{s:leastaction}

We begin our study by using Hamilton's least action principle,
and a variant of the standard Helmholtz decomposition of vector fields,
to provide a simple derivation of the governing equations
for smooth ideal fluid flows with pressureless free boundary and self-interaction energy.
We recall V. I. Arnold's classic use of least action to formally characterizes solutions
of the Euler equations for incompressible flows in a fixed domain in terms of geodesic
paths of diffeomorphisms.  

Let $\Omega_t\subset\R^d$ denote the domain occupied by the fluid at time $t$, and let $X$
denote the Lagrangian flow map, defined on the space-time domain $Q=\cup_t \Omega_t\times\{t\}$
so that 
\begin{equation}\label{e:Xflow}
\dot X (z,t) = v(X(z,t),t), \qquad X(z,0)=z\in\Omega_0
\end{equation}
for all $t$ in some interval $[0,\tbar]$. 
Here the velocity field $v$ is presumed to be sufficiently smooth up to the boundary.
The associated density field $\rho$ with given constant initial density $\rho_0$ is given by
\begin{equation}
\rho(x,t) = \rho_0\det\left(\frac{\D X}{\D z}(z,t)\right)\inv, \qquad x=X(z,t)\in\Omega_t.
\end{equation}
We let $\calA = \calK-\calV$ denote the Lagrangian action associated with the flow,
where  
\begin{align}
&\calK = \frac12\int_0^\tbar\int_{\Omega_t}\rho(x,t)|v(x,t)|^2\,dx\,dt 
= \frac12 \int_0^\tbar\int_{\Omega_0}\rho_0 |\dot X(z,t)|^2\,dz\,dt,
\label{d:calK}
\\&
\calV = \frac12 \int_0^\tbar \int_{\Omega_t^2} \Phi(x,x')\rho(x,t)\rho(x',t)\,dx\,dx'\,dt
\label{d:calV}
\end{align}
respectively denote kinetic energy and self-interaction energy with symmetric kernel $\Phi(x,x')$, given for the Newtonian gravitational potential in particular by 
\[
\Phi(x,x')= - \frac{G}{|x-x'|}.
\]

For any family $\eps\to X_\eps$ of flow maps depending smoothly on a variational parameter $\eps$,
one finds that the variation $\delta X = (\D X/\D \eps)|_{\eps=0}$ induces a density variation 
$\delta\rho$ satisfying 
\[
-\frac{\delta\rho}\rho = \nabla\cdot \tilde v, \qquad \tilde v(x,t) = \delta X(z,t),
\]
so naturally $\nabla\cdot\tilde v=0$ for variations that leave the density invariant.

We proceed to compute the variation of the action at a density-preserving flow for 
density-preserving variations. 
Firstly, requiring that the variation $\delta X$ \emph{vanishes at the endpoints $t=0$ and $\tbar$},
we find
\begin{align*}
\delta \calK &= 
\int_0^\tbar\int_{\Omega_0}\rho_0
\dot X(z,t)\cdot\delta\dot X(z,t)\,dz\,dt
= - \int_0^\tbar\int_{\Omega_0}\rho_0
\ddot X(z,t)\cdot\delta X(z,t)\,dz\,dt
\\&=
- \int_0^\tbar\int_{\Omega_t}\rho_0(\D_t v+v\cdot\nabla v)\cdot\tilde v \,dx\,dt ,
\\
\delta \calV &= 
-\int_0^\tbar\int_{\Omega_t}\rho_0 f(x,t)\cdot \tilde v(x,t) \,dx\,dt ,
\end{align*}
where $f(x,t)$ is the (specific) self-interaction force field, given by
\[
f = -\nabla \varphi, \qquad \varphi(x,t) = \int_{\Omega_t} \rho_0\Phi(x,x')\,dx'.
\]
Now a flow $X$ is critical for the action $\calA$ if the variation
\[
\delta A = \delta K-\delta V = 
-\int_0^\tbar\int_{\Omega_t}
\rho_0(\D_t v+v\cdot\nabla v -f)\cdot\tilde v\,dx\,dt = 0,
\]
for all virtual displacements $\tilde v$ for which $\nabla\cdot \tilde v=0$ in $\Omega_t$
and which vanish at $t=0$ and $\tbar$.
At this point we note that any $L^2$ vector field $u$ on $\Omega_t$ has a unique
$L^2$-orthogonal decomposition
\begin{equation}
u = w + \nabla p, \qquad\mbox{with}\quad \nabla\cdot w=0 \mbox{\ \ in $\Omega_t$}, \quad p = 0 \mbox{\ \ on $\D\Omega_t$},
\end{equation}
obtained by solving $\Delta p=\nabla\cdot u$ for $p$ in the Sobolev space $H^1_0(\Omega_t)$.
(This is a variant of the standard Helmholtz decomposition, see \cite[p.~215]{DL}.)
By choosing $u=f-(D_tv+v\cdot\nabla v)$, we infer that for a density-preserving critical path,
the velocity field should satisfy the Euler equations
\begin{equation}\label{e:euler}
\D_t v + v\cdot\nabla v + \nabla p = f, \quad \nabla\cdot v=0 \mbox{\ \ in $\Omega_t$},
\end{equation}
with the condition
\begin{equation}\label{pzero}
\quad p = 0 \mbox{\ \ on $\D\Omega_t$}
\end{equation}
on the free boundary, along with the \emph{kinematic} condition that $\Omega_t=X(\Omega_0,t)$.

It will be useful below to note that in terms of the deformation gradient $F = {\D X}/{\D z}$,
Euler's equations  \eqref{e:euler} in Lagrangian coordinates take the form
\begin{equation}\label{e:eulerL}
F^T \ddot X + \nabla\tilde p +\nabla\tilde\varphi=0, \qquad \det F = 1,
\end{equation}
where $\tilde p(z,t)=p(X(z,t),t)$ and $\tilde \varphi(z,t)=\varphi(X(z,t),t)$ 
respectively represent the pressure and force potential in Lagrangian coordinates, 
since by the chain rule, e.g., 
\[
\frac{\D \tilde p}{\D z}= \frac{\D p}{\D x}\frac{\D X}{\D z}
\qquad\mbox{so}\quad
\nabla \tilde p = F^T\nabla p.
\]

\section{Self-gravitation and Dirichlet's symmetry} 
\label{s:gravity}

In an effort to understand the shape of the earth and other celestial bodies,
many prominent investigators, starting with Isaac Newton, have studied the shape of a
rotating body of fluid with self-gravitation.  
Much historical information on this topic can be found in the book of Chandrasekhar~\cite{chandra}.

In particular, Dirichlet, in a posthumously published paper edited by Dedekind,
was the first to develop equations for \emph{time-dependent} motions 
that preserve ellipsoidal shape \cite{dirichlet1860}.
Of the numerous interesting developments following Dirichlet's work, we mention only a few.
From Dirichlet's equations, Dedekind explictly deduced a surprising symmetry, 
and used it to find ellipsoids with nontrival internal flows, 
conjugate to the rigidly rotating fluid ellipsoids discovered earlier by Jacobi.
In a remarkable paper, Riemann subsequently showed that all rotating, shape-preserving ellipsoids 
fall into three simple classes, and initiated a study of their stability by energy critera~\cite{riemann}.

The reason we bring up this subject is to describe how Dirichlet's ellipsoidal motions 
can be characterized by through a finite-dimensional least-action principle,
and to thereby provide a simple derivation of Dedekind's symmetry.
The first descriptions of a reduced least-action principle for Dirichlet ellipsoids 
appeared only a few years after Riemann's work, in papers by Padova~\cite{Padova1871} and Lipschitz~\cite{Lipschitz1874}; see the excellent review by Borisov et al.~\cite{borisov2009}. In the absence of gravitation, 
critical paths of action correspond to constant-speed geodesic motion
on a determinant-constrained surface in the space of matrices describing the deformation,
as noted by O. M.  Lavrenteva \cite{Lav80}. 

We proceed to details.  Following Dirichlet, we seek motions for which the domain
$\Omega_t\subset \R^3$ is ellipsoidal, with time-dependent semi-axes $a_j(t)$, $j=1,2,3$,
having a constant product $a_1a_2a_3$.
We require the Lagrangian map $z\mapsto X(z,t)$ to be linear, taking the convenient form
\begin{equation}
X_i(z,t) = \sum_{j=1}^3 P_{ij}(t) \frac{z_j}{a_j(0)},  \qquad z\in\Omega_0,
\end{equation}
or in more succinct matrix-vector form,
\begin{equation}\label{d:Xz}
X(z,t) = P(t)\Lambda_0\inv z, \quad \Lambda_t = \diag\{a_j(t):j=1,2,3\}.
\end{equation}

We presume the initial ellipsoid is $\Omega_0 = \Lambda_0 B_1$, where $B_1$ is the unit ball.
That is, $z\in\Omega_0$ if and only if $z=\Lambda_0 y$ with $y\in B_1$.
After a rotation of coordinates, $X(z,t)$ should lie in $\Lambda_t B_1$.
Thus there should exist orthogonal matrices $R(t)$ and $S(t)$ such that 
\begin{equation}\label{d:Xy}
X(z,t) = R(t)\Lambda_t S(t)^Ty, \quad\mbox{with $y=\Lambda_0\inv z$.}
\end{equation}
The modern eye will recognize that this provides the \emph{singular value decomposition}
\[
P = R\Lambda S^T, \quad RR^T=I=SS^T,
\]
with the semi-axes $a_j$ being the singular values of $P$.
The matrix $P(t)$ should satisfy 
\begin{equation}\label{e:P0}
P(0)=\Lambda_0 \quad\mbox{ and }\quad \det P(t)=a_1a_2a_3= \mbox{  constant.}
\end{equation}

The deformation gradient will be a function of time alone, taking the form
\begin{equation}\label{d:Ft}
F(t) = \frac{\D X}{\D z}(z,t) = P(t)\Lambda_0\inv.
\end{equation}
Substituting into Euler's equations written in Lagrangian coordinates,
we require
\begin{equation}\label{e:EulerF}
 F^T\ddot F z + \nabla \tilde p +\nabla\tilde\varphi=0,
\end{equation}
where $\tilde p$ here  is pressure divided by mass density.

It is a remarkable fact, due to Gauss and Rodrigues (see \cite{chandra,Fitzpatrick}) that the self-gravitation potential is quadratic in the spatial variables, taking the following form.
With respect to the coordinates $\hat x=R^Tx$ taken along the principal axes of the ellipsoid, 
\begin{equation}
\varphi(x,t) = - G\rho_0\pi\, a_1a_2a_3
\left(\alpha_0-\sum_{i=1}^3 \alpha_i \hat x_i^2\right),
\end{equation}
\begin{equation}\label{d:alphas}
\alpha_0 = \int_0^\infty \frac{du}{\Delta}, \qquad
\alpha_i = -\frac{1}{a_i}\frac{\D\alpha_0}{\D\alpha_i}= \int_0^\infty \frac{du}{\Delta(a_i^2+u)}, \qquad
\Delta^2 = \prod_{i=1}^3 (a_i^2+u).
\end{equation}
In Lagrangian variables, using \eqref{d:Xy} and noting $R^TX = \Lambda S^T\Lambda_0\inv z$ we may then write
\begin{equation} \label{d:tildephi}
    \tilde\varphi(z,t) = -G\rho_0\pi (\alpha_0+z^TQz)\det\Lambda
\end{equation}
where 
\begin{equation}
    Q = \Lambda_0\inv S \Lambda \frac{\D\alpha_0}{\D\Lambda} 
    S^T \Lambda_0\inv ,
    \quad \frac{\D\alpha_0}{\D\Lambda} = \diag\left\{\frac{\D\alpha_0}{\D\alpha_i}\right\}.
\end{equation}
Then the Lagrangian potential gradient is linear in $z$, with
\begin{equation}
    \nabla\tilde\varphi(z,t) = -2G\rho_0\pi(\det\Lambda)\, Qz
\end{equation}
In light of \eqref{e:EulerF}, the pressure must be quadratic in the spatial variables. In order to vanish on the ellipsoid boundary, it must therefore be that for some scalar function $\beta(t)$,
\begin{equation}
    \tilde p(z,t) = \frac12 \beta(t)(1-|\Lambda_0\inv z|^2)
    \quad\mbox{and}\quad
    \nabla\tilde p(z,t) = -\beta(t)\Lambda_0^{-2}z.
\end{equation}
Substituting the above expressions directly into \eqref{e:EulerF}, with $\gamma_0=2\pi G\rho_0 $ we find Dirichlet's result in the following form.
\begin{lemma}
The linear Lagrangian map in \eqref{d:Xz} provides a solution to the Euler equations if
$P(t)$ satisfies
\begin{align}\label{e:Ptt}
 P^T\ddot P = \beta(t)I + (\gamma_0\det\Lambda) S\Lambda 
\frac{\D\alpha_0}{\D\Lambda} S^T ,
\end{align}
along with the conditions in \eqref{e:P0}.
\end{lemma}

Next we want to show how \eqref{e:Ptt} arises from reduced least action,
and derive Dedekind's symmetry. Using the fact that
$ 3\int_{B_1} y_i^2\,dy = \frac{4\pi}5 $, the kinetic energy in \eqref{d:calK} is 
reduced to an expression in terms of $\dot P$ via
\begin{equation}\label{e:K(P)}
    \calK(P) = 
    \frac{2\pi}{15}  (\rho_0\det \Lambda_0)
    \int_0^\tbar 
    \tr(\dot P^T\dot P)
    \,dt
\end{equation}
The gravitational potential energy is reduced to an expression in terms of $P$ via
\begin{equation}
    \calV(P) = -\frac12 G  
    (\rho_0\det\Lambda_0)^2
    \int_0^\tbar\int_{B_1^2}
    \frac1{|P(y-y')|}\,dy\,dy'\,dt 
\end{equation}
Note that $\tr(\dot P^T\dot P)=\tr(\dot P\dot P^T)$,
and  since the singular value decomposition of $P^T$ is $S\Lambda R^T$,
orthogonal changes of variables in the last integral yields
\begin{equation}\label{d:U(P)}
 \int_{B_1^2} \frac1{|P(y-y')|}\,dy\,dy'
    =\int_{B_1^2} \frac1{|\Lambda(y-y')|}\,dy\,dy'
    =\int_{B_1^2} \frac1{|P^T(y-y')|}\,dy\,dy'
\end{equation}
By consequence we infer
\begin{lemma}
The reduced action $\calA(P)=\calK(P)-\calV(P)$ of every matrix path $P$ satisfies
\begin{equation}
    \calA(P) = \calA(P^T).
\end{equation}
\end{lemma} 
Since $P$ is a smooth function of $P^T$ and vice versa, 
the chain rule implies that $P^T$ is a critical path for the (determinant-constrained) 
action if and only if $P$ is.
This is \emph{Dedekind's symmetry}, which he used to discover that Jacobi's rigidly rotating ellipsoids correspond to ellipsoids with steady internal flows.

Lastly, we wish to indicate how the evolution equation \eqref{e:Ptt} arises by least action from the reduced action, due to the orthogonal invariance of the reduced potential energy.
\begin{lemma}\label{l:dUdP}
Given any $C^1$ function $\calU:\R^{m\times n}\to \R$ invariant with respect to both right and left multiplication by orthogonal matrices, 
its derivative at a matrix  $P$ can be expressed in terms of the 
singular value decomposition $P=R\Lambda S^T$, $\Lambda=\diag\{a_i\}$,
in the form
\[
\frac{\D U}{\D P} = R \frac{\D U}{\D\Lambda} S^T,
\]
where 
 \[
 \frac{\D U}{\D P} = \left( \frac{\D }{\D P_{ij}}U(P)  \right)
 \quad\mbox{and}\quad
 \frac{\D U}{\D \Lambda} = 
\diag\left\{ \frac{\D U}{\D P_{ii}}(\Lambda)\right\}.
 \]
\end{lemma}
\begin{proof}
By density we may assume the singular values $a_i$ of $P$ are distinct. 
Then for any perturbation direction $\tilde P$ there is a $C^1$-smooth singular value decomposition
\[
P+\eps \tilde P = R(\eps)\Lambda(\eps)S(\eps)^T
\]
for  $|\eps|$ small enough. Letting $'$ denote the derivative in $\eps$, 
evaluated at $\eps=0$, we note
\[
\Lambda' = R^T\tilde P S - R^TR'\Lambda - \Lambda S'^TS
\]
Since $\tr(AB)=\tr(BA)$ for any square matrices $A$ and $B$,
we then find by invariance that
\begin{align*}
& 
\tr\left(\frac{\D U}{\D P}^T \tilde P  \right) 
=  \left.  \frac{d}{d\eps} U(P+\eps\tilde P)  \right|_{\eps=0} 
=  \left.  \frac{d}{d\eps} U(\Lambda(\eps))  \right|_{\eps=0} 
= 
\tr\left(
\frac{\D U}{\D\Lambda}
\Lambda' 
\right)   
\\& =
\tr\left(
S\frac{\D U}{\D\Lambda}
R^T\tilde P 
\right) 
-
\tr\left(
R^TR'
\Lambda
\frac{\D U}{\D\Lambda}
\right) 
-
\tr\left(
\frac{\D U}{\D\Lambda}
\Lambda
S'^TS
\right) 
= 
\tr\left(
S \frac{\D U}{\D\Lambda}
R^T\tilde P 
\right) .
\end{align*}
The last equality holds because $R^TR'$ and $S'^TS$ are skew while 
$\Lambda\frac{\D U}{\D\Lambda}=  \frac{\D U}{\D\Lambda}\Lambda$ 
is symmetric.
\end{proof}

The reduced gravitational potential energy takes a classic expression \cite[p.~700]{Lamb} in terms
of the singular values $a_j$ of $P$, using 
the function $\alpha_0=\alpha_0(\Lambda)$ from \eqref{d:alphas}, as
\begin{align*}
\calV(P) &= 
\frac12 \rho_0  \int_0^\tbar
\int_{R\Omega_t}\varphi(R\hat x,t)\,d\hat x\,dt
=-\frac3{10} GM^2
\int_0^\tbar \alpha_0(\Lambda(t)) \,dt ,\qquad M = \rho_0\frac{4\pi}3a_1a_2a_3.
\end{align*}
Therefore the quantity in \eqref{d:U(P)} can be expressed in the form
\begin{equation}
 \int_{B_1^2} \frac1{|P(y-y')|}\,dy\,dy' = 
\frac{(4\pi)^2}{15} U(P),
\qquad\mbox{where}\quad
U(P)=U(\Lambda) =  \alpha_0(\Lambda).
\end{equation}
Incorporating the constraint $\log\det P(t)=$ const 
yields the augmented reduced action 
\[
\tilde\calA =  \frac{4\pi}{15}\rho_0\det\Lambda_0
 \int_0^\tbar  \left(\frac12\tr(\dot P^T\dot P) 
+\gamma_0(\det\Lambda_0)U(P)+\beta(t)\log\det P\right)\,dt 
\]
Applying Lemma~\ref{l:dUdP} after noting $\det P =\det\Lambda$,
we find that the criticality condition $\delta \tilde A=0$
subject to the constraint $\det P(t)=$ const corresponds to 
the equation
\begin{equation}
\ddot P 
=R\left(\beta(t)\Lambda\inv+(\gamma_0\det\Lambda)\frac{\D\alpha_0}{\D\Lambda}\right)S^T,
\end{equation}
which is equivalent to \eqref{e:Ptt}.
(It is curious that Lemma~\ref{l:dUdP} provides Abel's formula 
for the derivative of $\log\det P$.)

There is a considerable body of modern literature studying the Hamiltonian dynamics of the reduced dynamics; we refer to 
Borisov et al.~\cite{borisov2009},  Morrison et al. 2009\cite{morrison2009}, 
and Lewis~\cite{lewis2013} for further discussion and references.

\section{Dirichlet ellipsoids and hyperboloids}
\label{s:dirichlet}

Next we specialize the discussion to review properties of 
a family of simple exact solutions to the zero-gravity water wave equations
with pressureless free boundaries given by conics.
In particular we pay attention to the possible singular features of such flows,
focussing on 2D and the development of fluid jets.
We remark also upon a geodesic interpretation that proved useful in our
study \cite{LPS} that was motivated by a droplet splitting scenario.

The flows that we study here are all simple straining flows.
The ellipsoids are special cases of solutions found by 
Dirichlet \cite{dirichlet1860}, and hyperboloids
were found by Longuet-Higgins \cite{Longuet1972}.

\subsection{Geodesic curves of conics.}

We now describe some potential flows with conic free surface in any dimension $d\ge2$.
The Lagrangian flow map associated to the velocity field $v=\nabla\phi$ will satisfy
\begin{equation}\label{e:Xdot}
\dot X(z,t) = \nabla\phi(X(z,t),t),
\qquad
X(z,0)= z,
\end{equation}
for all $z\in \Omega_0\subset\R^d$ and all $t$.
All our flows here will correspond to quadratic potentials of the form
\begin{equation}\label{d:phi0}
\phi(x,t) = \frac12 \sum_{j=1}^d \alpha_j(t)x_j^2 -\lambda(t),
\qquad \mbox{with}\quad \Delta\phi=\sum_{j=1}^d \alpha_j(t)=0,
\end{equation}
so that the components of the Lagragian map evolve in a purely dilational way according to
\begin{equation}\label{e:dotXj}
\dot X_j = \alpha_j(t) X_j\,, \quad j=1,\ldots,d .
\end{equation}
Fixing some $\sigma_0\in\R$ and some choice of signs $\sigma_j=\pm 1$ for $j=1,\ldots,d$, 
the fluid will be taken to occupy a domain of the form 
\begin{equation}\label{d:Omegat}
\Omega_t=\{x\in\R^d: S(x,a(t))<\sigma_0\},
\end{equation}
where we define
\begin{equation}\label{d:Sxa}
S(x,a) = \sum_{j=1}^d \sigma_j \frac{x_j^2}{a_j^2} \,, 
\qquad
a=(a_1,\ldots,a_d)\in \R^d_+\,.
\end{equation}

The kinematic condition that the boundary flows with the fluid requires that for $z\in\D\Omega_0$,
\[
0 = \frac12 \frac{d}{dt}S(X,a) 
= \sum_{j=1}^d \sigma_j\frac{X_j^2}{a_j^2}\left(\alpha_j - \frac{\dot a_j}{a_j}\right)
\]
Leaving degenerate cases aside, it suffices to suppose that 
\begin{equation}\label{e:dotaj}
\dot a_j = \alpha_j a_j \,,\quad j=1,\ldots,d.
\end{equation}
Due to the incompressibility constraint in \eqref{d:phi0} it follows that the product
\begin{equation}\label{ed:c}
a_1\cdots a_d=r^d 
\end{equation}
remains constant in time. 

We recall the simple proof of the following result from \cite{LPmajda} (with a slight change of notation)
that provides a geodesic interpretation for solutions of the kind considered here.
\begin{proposition}\label{p:Edrop} Given a constant $r>0$,
let $a(t)=(a_1(t),\ldots,a_d(t))$ be any constant-speed geodesic on
the surface determined by the relation \eqref{ed:c}
in the space $\R_+^d$ with metric of signature $(\sigma_1,\ldots,\sigma_d)$
(possibly indefinite). 
Then this determines an ideal potential flow with $\Omega_t$ as in \eqref{d:Omegat}, 
pressure given by 
\begin{equation}\label{d:pressure}
p(x,t) =  \frac{\beta(t)}2(\sigma_0- S(x,a)), 
\qquad 
\beta(t) = \frac{\sum_j \dot a_j^2/a_j^2}{\sum_j \sigma_j/ a_j^2 } ,
\end{equation}
and potential $\phi$ given by \eqref{d:phi0} with $\alpha_j=\dot a_j/a_j$
and $\dot\lambda = \frac12\beta\sigma_0$.
\end{proposition}

\begin{proof} The path $t\mapsto a(t)$ 
is a geodesic on the surface defined by \eqref{ed:c} with constant squared speed 
$\sum_j \sigma_j\dot a_j^2$
if and only if the acceleration $\ddot a$ is parallel to the surface normal. 
Here this means that there is some scalar $\beta=\beta(t)$, 
\begin{equation}\label{e:ddotaj}
\ddot a_j = \frac{\beta\sigma_j}{a_j}, \qquad j=1,\ldots,d.
\end{equation}
The reason is that such a geodesic is a critical
path for the augmented action
\[
\tilde\calA = \int_0^T \sum_j \left(\frac12 \sigma_j\dot a_j^2 + \beta(t)\log \frac{a_j}r\right)\,dt.
\]
The value of $\beta(t)$ must be as stated in \eqref{d:pressure} since we require
\[
0 = \frac{d^2}{dt^2} \sum_{j}\log a_j = \sum_j \frac{a_j\ddot a_j - \dot a_j^2}{a_j^2}.
\]

Define $\phi$ by \eqref{d:phi0} with $\alpha_j$ and $\dot\lambda$ as stated in the Proposition.
Because $ \dot\alpha_j+\alpha_j^2= \ddot a_j/a_j=\beta\sigma_j/a_j^2$, 
the pressure from the Bernoulli equation \eqref{1.bernoulli} must satisfy
\begin{align*}
p &= -\phi_t-\frac12|\nabla\phi|^2 = \dot\lambda -\frac12\sum_j (\dot\alpha_j+\alpha_j^2)x_j^2
=  \frac\beta2(\sigma_0- S(x,a)).
\end{align*}
Thus $p=0$ on $\D\Omega_t$, and the ideal droplet equations all hold. 
\end{proof}

Under the present conventions, we note that
the Taylor sign condition \eqref{e:taylor} holds exactly when 
$p>0$ in $\Omega_t$, and this occurs exactly when
$\beta>0$ in \eqref{d:pressure}.

\subsection{Ellipsoidal droplets}\label{ss:ellipse}
The fluid domains $\Omega_t$ always remain bounded and ellipsoidal 
in case $\sigma_j=1$ for all $j=0,1,\ldots,d$.
These Dirichlet ellipsoids played an important role in the study of action-infimization
for free boundary droplet flows carried out in \cite{LPS},
particularly the ones corresponding to length-minimizing paths. 

The solution remains smooth globally for $t\in\R$, since the vector $a(t)$ of semi-major axis lengths
moves at a constant (Euclidean) speed $c=|\dot a|$ on the surface \eqref{ed:c} 
and cannot reach any singular point in finite time.
The pressure $p>0$ in $\Omega_t$
because $\beta>0$ in \eqref{d:pressure}, so the Taylor sign condition holds,
consistent with well-known results on well-posedness for water wave dynamics
\cite{Wu97,Wu99,Lindblad,CoutShko2007}.  

Each velocity component $\dot a_j$ is increasing,
because it turns out that $\ddot a_j=\beta\sigma_j/a_j>0$ for all $j$. 
The speed $c$ bounds $|\dot a_j|$ for all $j$ as well.
As $t\to+\infty$, necessarily some component $a_j\to\infty$,
and as $t\to-\infty$, some component $a_k\to\infty$, since 
$\sum \dot a_j/a_j=0$. 

\subsection{Ellipsoidal voids} The fluid can be considered to occupy the domain
\emph{exterior} to the ellipsoids above by taking $\sigma_j=-1$ for all $j$.
In this case, the pressure $p<0$ in $\Omega_t$ because $\beta<0$ in \eqref{d:pressure}.
The Taylor sign condition fails by consequence, and
we can expect this `bubble' flow to be highly unstable.

\subsection{Hyperbolas in 2D}
For the case when the signs of $\sigma_j$ can differ, 
the planar case $d=2$ admits the most simple and complete description.
We set $\sigma_0=\sigma_1=-1 = -\sigma_2$, so that the domain $\Omega_t$ corresponds to 
\begin{equation}
\frac{x_1^2}{a_1^2} > 1 + \frac{x_2^2}{a_2^2}.
\end{equation}
The equations of motion derive solely from 
incompressibility and geodesic speed constraints:
\begin{equation}
a_1a_2 = r^2,\qquad  -\dot a_1^2 + \dot a_2^2 = \hat s \in\R.
\label{e:a1a2}
\end{equation}
Eliminiating $\dot a_2$ we find  $\dot a_1^2 (a_2^2-a_1^2) = \hat s a_1^2$, 
whence 
with $\tau=\pm\sqrt{|\hat s|}$ we have
\begin{equation}
\dot a_1 
= \frac{\tau}{|\tan^2\theta-1|^{1/2}}\,,
\qquad \tan\theta = \frac{a_2}{a_1} = \frac{r^2}{a_1^2}\,. 
\label{ev:a1dot}
\end{equation}
Here $\theta=\theta(t)$ is the angle that the hyperbola's asymptote makes with the $x_1$ axis.

The pressure from \eqref{d:pressure} has the same sign as $\beta$,
which is given here by 
\[
\beta = \frac{ a_2^2 \dot a_1^2+a_1^2\dot a_2^2}{a_1^2-a_2^2} = \frac{2\dot a_2^2}{1-\tan^2\theta}.
\]
The pressure is positive and the Taylor sign condition \eqref{e:taylor} holds 
when $0<\theta<\pi/4$ ($a_1>a_2$), 
and pressure is negative and the Taylor sign condition violated
when $\pi/4<\theta<\pi/2$ ($a_1<a_2$).

{\it Singularities.}
No solution exists globally for $t\in\R$.  The solution becomes singular in
finite time when $a_1-a_2$ reaches zero, which means that the
asymptotic angle $\theta$ reaches $\pi/4$. 
If initially $\theta<\pi/4$ and $\dot a_1<0$ 
the solution becomes singular as $t$ increases, 
but exists globally for $t<0$ with $a_1\to\infty$ as $t\to-\infty$.  
The same happens if $\theta>\pi/4$ and $\dot a_1>0$.
The reverse happens if $\theta<\pi/4$ and $\dot a_1>0$,
or if $\theta>\pi/4$ and $\dot a_1<0$---the solution exists globally
for $t>0$ with $a_1\to\infty$ as $t\to+\infty$.

In all cases, the {\em free surface shape remains smooth} approaching a singular time.
If the Taylor sign condition holds and $t$ increases approaching
singularity, the angle between the asymptotes widens and approaches $90^{\circ}$.
The pressure and fluid velocity blow up {\em everywhere}, 
since $\alpha_1=\dot a_1/a_1$ blows up.  
Of course, the domain is unbounded and the energy is infinite, so 
it is unclear whether this is relevant for any finite energy flow.

{\it Corners.}  No free-surface singularity occurs in any solution we have discussed so far.
As Longuet-Higgins \cite{Longuet1980} pointed out, one obtains a simple flow 
containing a corner for all time, in a limit obtained by ``zooming out.'' 
Here it corresponds to taking $\sigma_0=0$, so that for example $\Omega_t$ corresponds
to the sector of the plane where
\[
\frac{x_1}{a_1(t)}>\frac{|x_2|}{a_2(t)}
\]
The same equations \eqref{e:a1a2} and 
\eqref{ev:a1dot} govern the evolution of the sector opening angle.
As above, the Taylor sign condition holds if the corner angle 
$2\theta$ is less than $90^\circ$ and is violated if it is greater 
than $90^\circ$.  Blowup occurs in the same ways as before. 

The condition $2\theta<90^\circ$ is consistent
with the theory for water waves with persistent corners
developed by  Kinsey and Wu \cite{KinseyWu2018} and Wu~\cite{wu2015blow},
since corners with angles less than $90^\circ$ have the
finite ``energy'' defined in \cite{KinseyWu2018} necessary to apply their theory.

\section{Locally singular ballistic interfaces}
\label{s:ZK}

Recently, Zubarev and Karabut \cite{ZK2018} and Zhuravleva et al.~\cite{zhuravleva2020new} 
have described rather explicit examples of ideal fluid flows on unbounded fluid domains 
that are capable of developing local singularities on the free surface.
These examples provide solutions of the ideal droplet equations
\eqref{1.laplace}--\eqref{1.pzero} 
that are derived from particular holomorphic solutions of the complex Hopf equation 
or inviscid Burgers equation
\begin{equation}\label{e:hopf}
V_t + V V_z = 0 \quad \mbox{for $z\in\Omega_t$.}
\end{equation}
Here $z=x+iy\in\Omega_t\subset\C$ corresponds to Eulerian variables in the fluid domain. 

A solution of \eqref{e:hopf} corresponds to a solution of \eqref{1.laplace}--\eqref{1.pzero} via 
\begin{equation}\label{d:ZK1}
V= u-iv=\phi_x-i\phi_y \,,
\end{equation}
{\em provided} (i) $V$ is holomorphic in $z$ on $\Omega_t$ (ii) the pressure-velocity relation
\begin{equation}\label{e:zkp}
p = - v^2 \quad\mbox{in $\Omega_t$}
\end{equation}
holds, and (iii) the pressure vanishes on $\D\Omega_t$, i.e., \eqref{1.pzero} holds.
This last condition means that fluid particles on the boundary move purely horizontally, 
and indeed the boundary must satisfy
\begin{equation}\label{e:imU}
\im V = 0 \quad \mbox{on $\D\Omega_t$.}
\end{equation}
As one can verify by straightforward computation, 
the real and imaginary parts of the Hopf equation yield Euler's equations,
noting $p_x = -2vu_y$, $p_y=-2vv_y$.

The characteristic curves $Z(t)$ for the Hopf equation are straight lines that satisfy
\[
\frac{d Z}{d t} = V(Z(t),t), \qquad  \frac{d}{dt} V(Z(t),t)= 0 .
\]
When $v\ne0$, these curves are not fluid particle paths.
However, {\em on the free surface where $p=0$ they are particle paths}. 
Consequently, particle paths on the surface evolve 
{\em in straight lines, horizontally at constant speed}.

In \cite{ZK2018}, the authors find solutions 
by solving implicitly characteristic equations in the form
\begin{equation}\label{e:zF}
z = Vt + F(V).
\end{equation}
Here $F(V)\to0$ as $V\to\infty$ for the values of $V$ relevant to the solution,
and $F$ should be chosen to avoid singularities when $z$ is in the fluid domain.
The case $F(V)=1/(V+i)$ is the simplest one that provides local singularities.
In this case one can use the horizontal velocity $u$ to parametrize 
the free surface via
\begin{equation}\label{e:zgplot}
z = tu + \frac{1}{u+i} = tu+\frac{u}{u^2+1} - \frac{i}{u^2+1}, \quad u\in\R.
\end{equation}
We plot this surface for $t=-4,-3,-2,-1$ in Fig.~\ref{f:ZGplot}.
The surface is a smooth graph $y=\gamma(x,t)$ for $t<-1$, 
since $dx/du<0$ for all $u$.
A cusp develops at $t=-1$, having $y\sim-1+|x|^{2/3}$. 

\begin{figure}
\includegraphics[width=3.5in]
{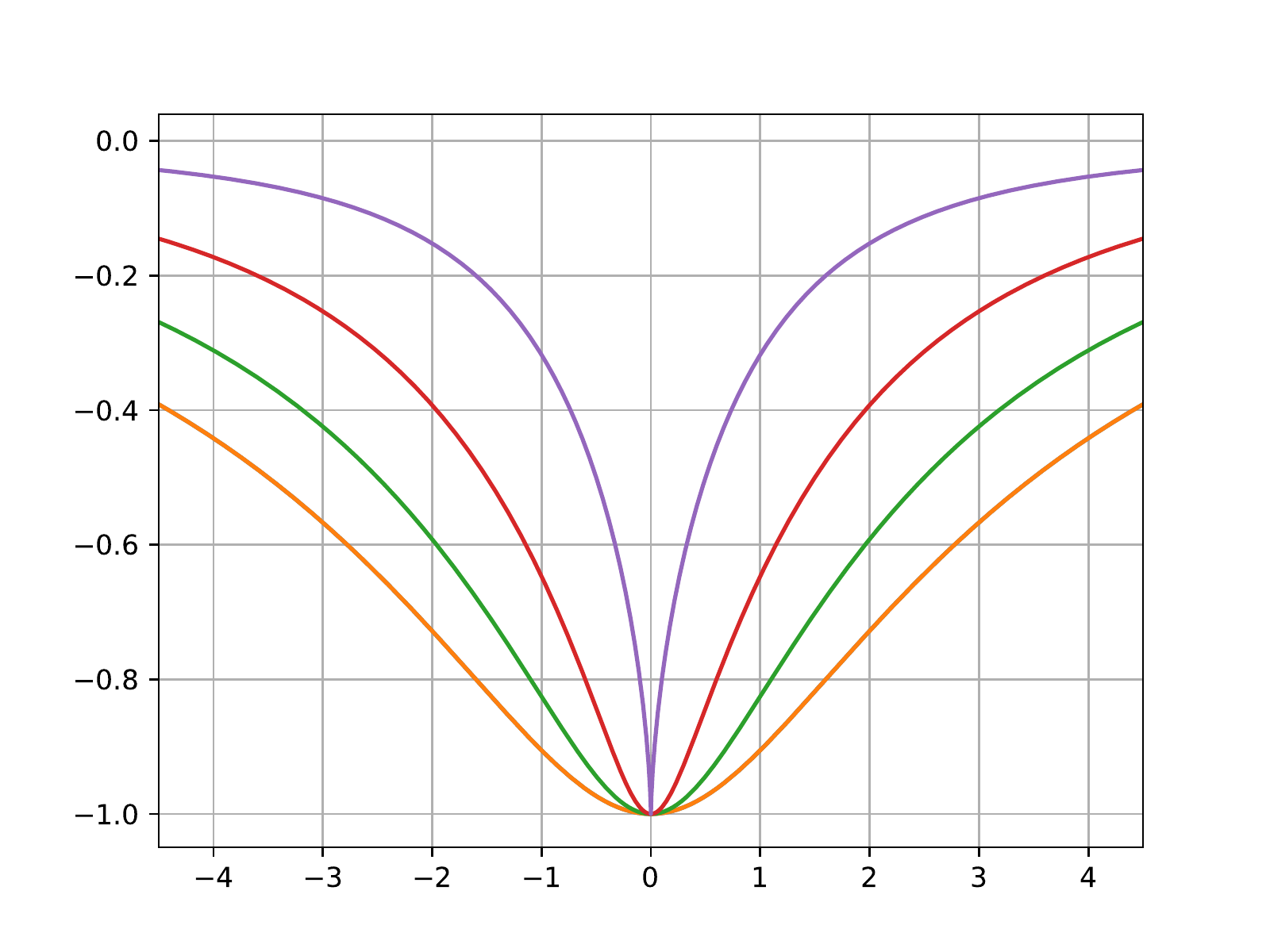}
\put(-126,2){\large $x$}\put(-247,93){\large $y$}
\put(-52,96){\vector(-1,0){30}}
\put(-195,96){\vector(1,0){30}}
\caption{The interface in \eqref{e:zgplot} for $t=-4,-3,-2,-1$ (from bottom to top)}
\label{f:ZGplot}
\end{figure}
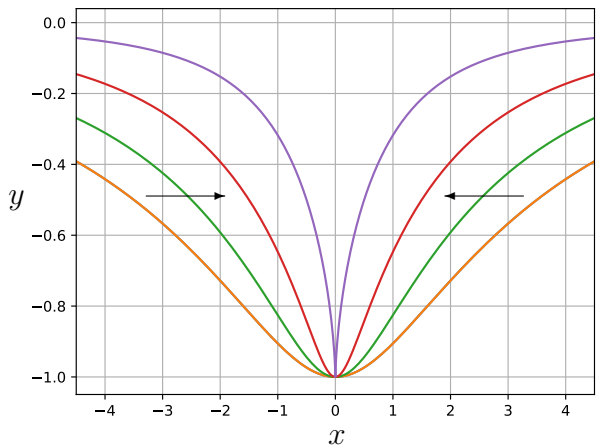

Very recently, Zhuravleva et al.~\cite{zhuravleva2020new} have described a different family
of solutions of the ideal droplet equations that describe unbounded flows surrounding a 
collapsing cavity. They use holomorphic solutions to the complex Hopf equation \eqref{e:hopf}
to determine fluid velocity in a different way, namely
\begin{equation}\label{d:ZK2}
 u-iv = \frac1V,
\end{equation}
and impose a different pressure-velocity relation, namely
\begin{equation}\label{zkp2}
 p = \frac12 \log(u^2+v^2) -\frac{u^2+v^2}2 + \frac12
\end{equation}
On the fluid boundary in this case, vanishing pressure necessitates the condition
\begin{equation}
|V|=1 \quad\mbox{for $z\in\D\Omega_t$.}
\end{equation}
Then on the free surface, one finds ballistic particle paths that coincide with 
characteristics according to the relations
\begin{equation}\label{e:Vinv}
z = (u+iv)t + z_0 = Vt + G(V) \,.
\end{equation}
The fluid interface can be determined parametrically by using \eqref{e:zF} with the relation 
$V=e^{i\theta}$ 
on the fluid boundary.  Corresponding to the choice
\begin{equation}\label{e:Fcavity}
G(V) = \frac{4a V}{1-b^4V^4},  \quad a=-0.2,\quad b=1.2,
\end{equation}
the authors in \cite{zhuravleva2020new} show that the cavity collapses to a splash singularity, as shown in 
the left panel of Fig.~\ref{f:cavity}, 
where the interface is shown at times $t=-3,-2,-1.03$ as in \cite{zhuravleva2020new}.  
In the right panel, we take $a=0.2$ instead and plot at the times $t=-11,-8,-5,-2$.
The figure indicates that a local singularity forms at a time $t\approx -5$ and loses physical meaning
after a self-intersection appears.  
Indeed, a local singularity must appear at the time $t=-G'(1)\approx -5.01176$
when the boundary parametrization degenerates.
(For sufficiently large negative times, $\partial z/\partial V\ne0$ for $|V|=1$ and 
injectivity of the map $V\mapsto z$ for $|V|<1$ follows by classical criteria, 
see section~\ref{ss:criteria} below.)

\begin{figure}
\includegraphics[width=3.0in]
{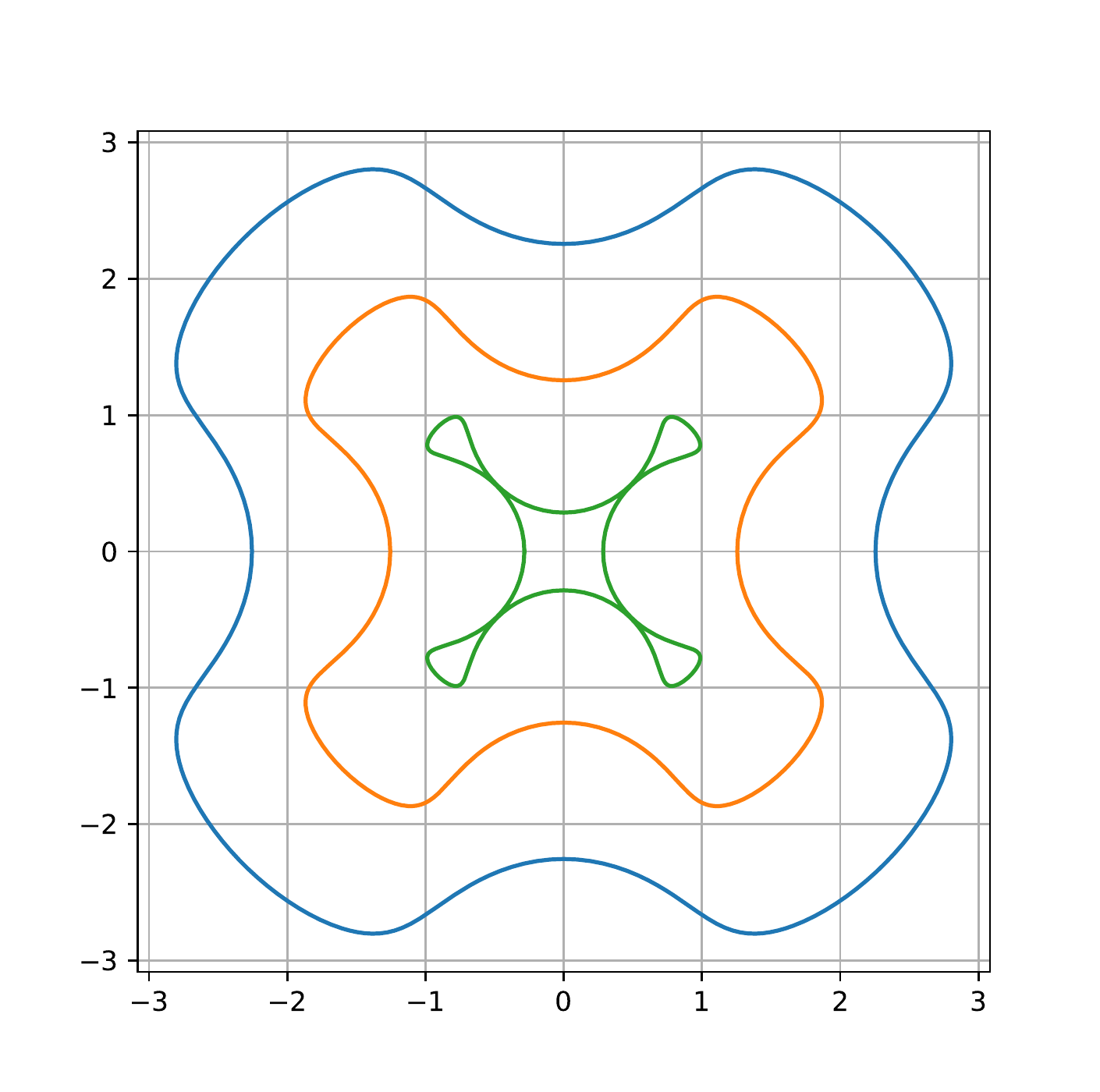}
 \includegraphics[width=3.0in]
{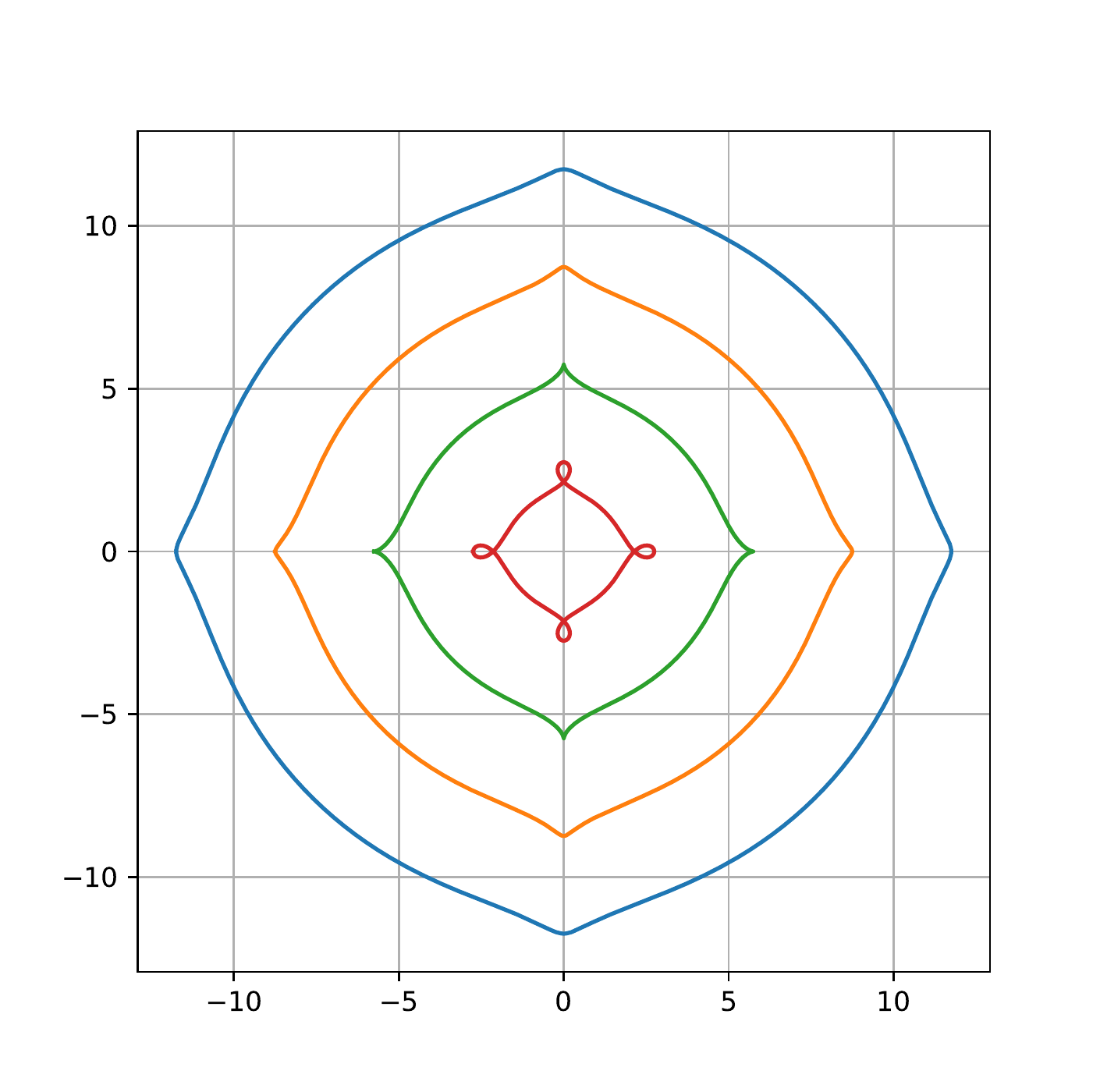}
\caption{Collapsing cavity with splash and local singularities. \newline
(Left: $a=-0.2$, $t=-3,-2,-1.03$. Right: $a=0.2$, $t=-11,-8,-5,-2$.) }
\label{f:cavity}
\end{figure}

For all of the singular solutions found in \cite{ZK2018} and \cite{zhuravleva2020new}, 
the fluid particles on the free surface experience zero acceleration. Indeed, the gradient
of the pressure vanishes at the free surface 
in each of the respective cases \eqref{e:zkp} and \eqref{zkp2}.
By consequence we have
\begin{equation}
\frac{\D p}{\D n} =0  \quad\mbox{on $\D\Omega_t$},
\end{equation}
so the strict Taylor sign condition \eqref{e:taylor} does not hold. 
Necessarily, $p<0$ inside the fluid domain, in fact, for both cases \eqref{e:zkp} and \eqref{zkp2}.

While the solutions in \cite{ZK2018} are certainly interesting, then,
it seems difficult to imagine how they
might approximate solutions of \eqref{1.laplace}--\eqref{1.pzero} in bounded domains,
since for the latter, $-p$ is always subharmonic due to \eqref{1.bernoulli}, 
so $p>0$ in the fluid domain.

\section{Numerical evidence for 2D local singularities}
\label{s:numerics}

In this section our goal is to study the possible development of local singularities
in smooth ideal potential flows through the use of several numerical illustrations and experiments.

Initially we expected that with zero gravity and surface tension, corners in
the free surface would form rather easily, as the fluid `tries to move ballistically' 
except for the pressure term that maintains incompressibility.  
As illustrated in the first examples below, however, our experience is consistent 
with the observations and remarks of Longuet-Higgens \cite{Longuet1983}, 
who used three-dimensional Dirichlet hyperboloids to 
explain jets in several kinds of fluid experiments,
and argued that such hyperboloidal jets may be characteristic of 
other types of unsteady free-surface flows.

As we illustrated in \cite{LPmajda}, it is not difficult to find and compute flows 
that exhibit a splash singularity, with interface that self-intersects at some positive time.
By varying parameters, we attempted to find flows with local singularities forming as
self-intersection points merge together.
But instead we found a tendency for strongly curved interfaces to be unstable through
the formation of small-scale (presumably hyperbolic) jets.

In \cite{LPmajda}, this led us to consider the expedient of exploiting the time-reversal
symmetry of the Euler equations.  We computed solutions \emph{expanding away} 
from a corner. Starting with a sequence of smooth approximations to a nonsmooth 
fluid domain, our computations suggested convergence to a smooth interface 
with bounded curvature at positive time. 

In the last subsection below, we extend these computations using 
equations \eqref{e:QU} instead of \eqref{e:ZF}, and with different initial data.
The results are consistent with the previous ones in \cite{LPmajda}, and are suggestive of
a self-similar scaling hypothesis for a two-parameter family of smooth solutions 
starting from an interface formed by an infinite wedge with power-law initial velocity.
We also provide a heuristic explanation of the scaling exponents that are observed here
and were first seen in \cite{LPmajda}.

\subsection{Conformal formulations and a pseudospectral scheme}
We perform our computations using a filtered pseudospectral discretization 
of the equations of motion in a conformal formulation.
An advantage of this approach that is well known is that the Dirichlet-to-Neumann map
for the fluid domain is replaced by that for the reference domain, which is easier to compute.

For the case that we study here, we will take the reference domain to be the unit disk $\bbD\subset\C$.
With this choice we can make use of the M\"obius automorphisms of $\bbD$ to concentrate grid points 
in some zone of high curvature.  An analogous transformation for periodic water waves was
described in \cite{Lushnikov2017}.  This method is convenient, but is limited in its capability
to resolve fine-scale flow features, as compared to more flexible 
boundary integral methods with adaptive grid refinement, say.

{\it Formulations.}
We refer to the appendix for a detailed derivation of the two conformal formulations 
that we make use of.  Briefly, we let $z=x+iy$ denote
complexified Eulerian coordinates in the fluid domain $\Omega_t\subset\C$.
This domain is assumed to be parametrized by a conformal map 
$w\mapsto \mb Z(w,t)$, $w\in\bbD$.
The boundary $\D\Omega_t$ is then parametrized by $\theta\in\bbT=\R/2\pi\bbZ$
via
\[
z = Z(\theta,t) := {\mb Z}(e^{i\theta} ,t) ,
\quad \theta\in\bbT.
\]
Since $Z = X+iY$ provides the boundary values of a holomorphic function 
in $\bbD$, the real part determines the imaginary part by the
Hilbert transform. With the expansion
\begin{equation}\label{e:Zk}
Z = \sum_{k\in\bbZ} \hat Z_k(t) e^{ik\theta}, \quad Z_k = X_k+iY_k,
\end{equation}
we have (presuming $\hat Y(0,t)=0$ for convenience)
\[
\mbox{$Y=HX$,\quad meaning}\quad
\hat Y_k(t) = (-i\sgn k)\hat X_k(t).
\]

The first conformal formulation involves 
$\bm Z(w,t)$, the conformal parametrization of the fluid domain, 
and $\bm F(w,t)$, the complex velocity potential.
Under the simplest conditions 
that uniquely fix the fluid parametrization, which are 
\begin{equation}
\frac{d}{dt} \bm Z(0,t) = 0, \qquad \frac{d}{dt}\arg \bm Z_w(0,t) = 0,
\end{equation}
the evolution equations for these quantities take the following form:
\begin{align}
\label{e:ZF}
\bm Z_t &= \bm Z_w \bm G, 
\qquad
\bm F_t = \bm F_w \bm G - \bm R,
\end{align}
where the traces $G, R$ of the holomorphic functions $\bm G, \bm R$ are respectively given by
\begin{align}
\label{d:GR}
G &= w(I+iH) \re\left(\frac Un\right), 
\qquad R =  (I+iH)\left(\frac12{|U|^2} \right).
\end{align}
Here surface pressure and body forces have been taken as zero.
In these expressions, $U$ and $n$ are the traces of the (anti-holomorphic)
velocity ${\bm U}$ and (unnormalized) normal vector $\bm n$, given by 
\begin{equation}
\bar{\bm U} = \frac{\bm F_w}{\bm Z_w}, \qquad \bm n = w\bm Z_w.
\end{equation}

We make use of a second conformal formulation in order to study dynamics in 
a very large domain approximating an infinite wedge.
The holomorphic function
\begin{equation}
\bm Q = \frac 1{\bm Z_{w}},
\qquad
\end{equation}
evolves together with $\bar{\bm U}$  according to the equations
\begin{align}
\label{e:QU}
\bm Q_t &= \bm Q_w \bm G - \bm Q \bm G_w \,,
\qquad
\bar{\bm U}_t = \bar{\bm U}_w \bm G - \bm Q \bm R_w \,,
\end{align}
with the traces of $\bm G$ and $\bm R$ given as in \eqref{d:GR}.
Essentially this same formulation was described by A.~I.~Dyachenko in \cite{dyachenko2001dok}
and was used recently by S.~A.~Dyachenko \cite{dyachenko2019jfm} 
to compute bounded ideal droplet solutions with and without surface tension.

In each of the two formulations, we compute by evolving just the real parts 
of the traces and determining the imaginary parts using the Hilbert transform.
To recover the boundary parametrization $Z$ from the second formulation 
in a nonsingular way for large domains not encircling $0$,
it is sometimes convenient to write
\begin{equation}
\label{d:S}
\bm S = \frac 1{\bm Z} 
\end{equation}
(or some other analytic function of $1/\bm Z$)
and evolve $\bm S$ (actually the real part of its trace) along with \eqref{e:QU} according to 
\begin{equation}
\label{e:S}
\bm S_t = \bm S_w \bm G \,. 
\end{equation}
When $1/\bm Q$ is not singular, we recover $Z$ by integrating with respect to $w$ 
using the fast Fourier transform as indicated below.

{\it Verification of \eqref{e:QU}.} For completeness we derive \eqref{e:QU} from \eqref{e:ZF}.
Since $\bar{\bm U}=\bm Q\bm F_w$, we get
\[
\bm Q_w = -\bm Q^2 \bm Z_{ww}\,, 
\qquad \bar{\bm U}_w = \bm Q \bm F_{ww} + \bm Q_w \bm F_w\,,
\]
\[
\bm Z_{wt} = \bm Z_{ww} \bm G + \bm Z_w \bm G_w\,,
\qquad \bm F_{wt} = \bm F_{ww} \bm G + \bm F_w \bm G_w-\bm R_w\,.
\]
Then it follows
$ 
\bm Q_t = -\bm Q^2 \bm Z_{wt}  =  \bm Q_w \bm G - \bm Q \bm G_w 
$ 
and 
$\bar{\bm U}_t = \bm Q_t \bm F_w + \bm Q \bm F_{wt}$, so 
\begin{align*}
\bar{\bm U}_t &= (\bm Q_w \bm G- \bm Q \bm G_w)\bm F_w + \bm Q( \bm F_{ww} \bm G + \bm F_w \bm G_w - \bm R_w)
= \bar{\bm U}_w \bm G - \bm Q \bm R_w \,.
\end{align*}

{\it Discretization.}
We use a straightforward psuedospectral scheme to
discretize the equations in space, using grid points $\theta_j=jh$,
$j=1,\ldots,N$, $h=2\pi/N$.  For the system \eqref{e:ZF}, 
we first expressed the equations in real form in terms of the operator
$\D_\theta = iw\D_w$, and then filter all derivatives by 
replacing $\D_\theta$ with Fourier symbol $ik$ 
by $\calD_\rho$ with Fourier symbol
\[
\hat\calD_\rho(k) = ik\,\rho(hk), \qquad \rho(\xi)=
\exp(-10(\xi/\pi)^{15})
\]
This filter is similar to that used in \cite{HouLi07}.
We use a standard ODE solver in the julia OrdinaryDiffEq package 
for time integration, with tolerance set to $10^{-9}$ or smaller.

For system \eqref{e:QU} we convert real parts to complex analytic form 
by the discrete Hilbert transform, e.g., representing $Q(\theta_j,t)$,
$j=1,\ldots,N$ by
\begin{equation}
Q_j(t) = \sum_{k=0}^{N/2-1} \hat Q_k(t) e^{ik\th_j} \,,
\end{equation}
then compute filtered derivatives by using the fast Fourier transform to evaluate
\begin{equation}
(\calD_\rho Q)_j = \sum_{k=1}^{N/2-1}  ik\rho(hk) Q_k e^{i(k-1)\th_j} .
\end{equation}
We sometimes found it useful for numerical stability to additionally filter 
the solution after each time step. 
We recover the interface position when $1/Q$ is nonsingular using the formula
\begin{equation}
Z_j(t) = \sum_{k=1}^{N/2-1} \frac{c_{k-1}(t)}k e^{ik\th_j} \,,
\end{equation}
assuming $\bm Z(0,t)=0$, 
where the coefficients $\hat c_k(t)$ are the discrete Fourier coefficients of $1/Q$.

{\it Accuracy check.}
We checked the accuracy of the numerical scheme for a Dirichlet ellipse 
as described in section~\ref{ss:ellipse} above, 
with initial data for \eqref{e:ZF} given by 
\begin{equation}
Z(\th,0) = e^{i\th}, \qquad F(\th,0)= e^{2i\th}.
\end{equation}
This corresponds to $a_1(0)=a_2(0)=1$ in \eqref{d:Omegat} and
$\alpha_1(0)=1$, $\lambda(0)=0$ in \eqref{d:phi0}.
To check how closely the solution conforms to an ellipse, we use
an explicit conformal map from ellipse to disk given by 
\begin{equation}\label{e:Nehari}
z=x+iy \mapsto w=\calW_q(z) := \sqrt{k(q)}\sn\left(\frac{2K}\pi \sin\inv z;q\right),
\quad q = \left(\frac{a-b}{a+b}\right)^2\,.
\end{equation}
Here $\sn$ is the Jacobi elliptic function with parameters $q$, $k(q)$, and $K=K(q)$, 
with notation as in \cite[p. 296]{Nehari}.
To evaluate this function, we ported the Matlab routine ELLIPJI by I.
Moiseev~\cite{elliptic} to julia.
For each system \eqref{e:ZF} and \eqref{e:QU}, we tabulate in Table~\ref{t:errs} 
the maximum pointwise error $E=E_{ZF}$ or $E_{QU}$ respectively, given by 
\[
E = \max_j |\calW_q(Z_j) - e^{i\th_j}|
\]
at time $t=0.25$, assuming the value $a=\frac1b\approx 1.278$ in \eqref{e:Nehari} is
given by $\re Z_j(t)$ with $j=0$ from the computed solution.
We specified a tolerance of $10^{-12}$ to the ODE solver for these computations.

\begin{table}
\begin{tabular}{|l | l | l|}
\hline
N & $E_{ZF}$  & $E_{QU}$ 
\\ [0.5ex] \hline
64 & 7.945e-4 & 9.475e-4
\\ \hline
128 &  8.078e-6 & 8.315e-6
\\ \hline
256 & 8.576e-9 & 3.175e-9
\\ \hline
512 & 1.342e-12 & 2.057e-13
\\ \hline
1024 & 1.898e-11 & 3.096e-14
\\ \hline
\end{tabular}
\medskip
\caption{Maximum-norm position errors for elliptical test case at $t=0.25$}
\label{t:errs}
\end{table}

\subsection{Examples with developing jets}

\subsubsection{Initial velocity with five-fold symmetry}
In our first example we take the initial shape as a circle, 
with five-fold symmetric initial velocity, corresponding to 
\[
\bm Z(w,0) = w,  \quad F(w,0) = -0.15 w^5, 
\quad \bm Q(w,0)=1, \quad \bar{\bm U}(w,0)= -0.75 w^4.
\]
We computed the solution from \eqref{e:QU} with $N=2^{14}$ grid points and plot
the solution along with a quiver plot of velocity,
at time $t=0.3$ in Fig.~\ref{fig5fold},
\begin{figure}
\includegraphics[width=4.5in]
{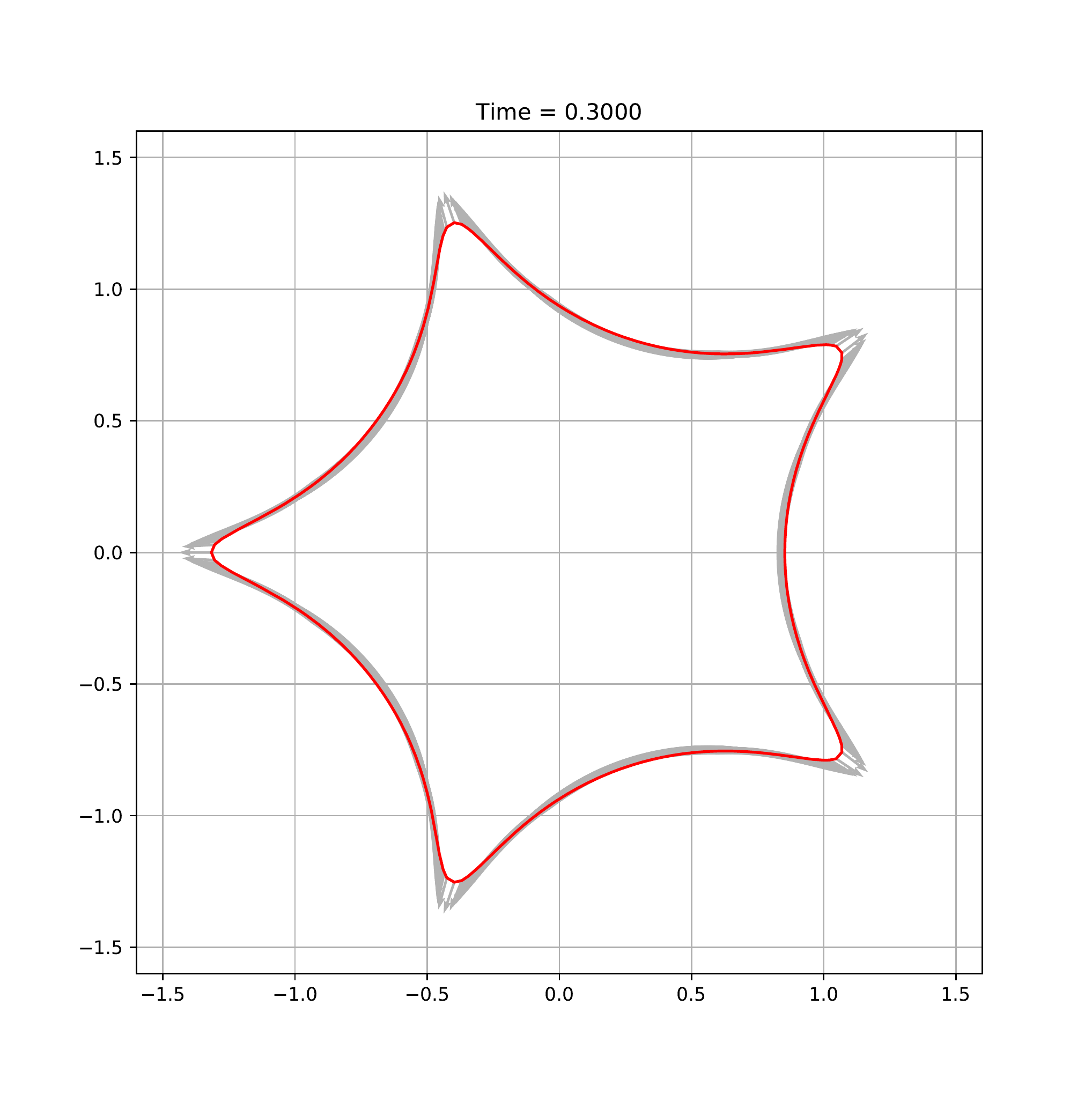}
\vspace{-0.8cm}
\caption{Interface from five-fold symmetric initial velocity}
\label{fig5fold}
\end{figure}
The interface shows the development of regions of high curvature
that may be incipient jets or corners.  The arrows, which are plotted 
at consecutive grid points, indicate that the uniformly spaced grid
on the parametrizing circle is being stretched severely as the jets develop.

\subsubsection{Initial velocity with single mode}

To study whether the protruberances that develop in the previous example 
might develop into corners, in \cite{LPmajda} we considered an initial velocity 
that produces a single tip.  This allows us to use a M\"obius transformation
to concentrate grid points in the single region of high curvature
and resolve the computation for a longer time.
Thus we solve equations \eqref{e:ZF} with initial data satisfying
\begin{equation}\label{e:nosedata}
{\mb Z}(w,0) = \zeta_r(w):= \frac{w+r}{1+rw}\,,
\qquad \re F(w,0) = 
  \left(\frac{\re \mb Z(w,0)+1}2\right)^5\,.
\end{equation}
Corresponding to compressing the grid by the factor
\begin{equation}\label{d:compress}
c = \left( \frac{1+r}{1-r}\right)^2 = 250 ,
\end{equation}
we take $r\approx0.881$.  Figure~\ref{fig1mod}, taken from \cite{LPmajda},
shows the interface computed at time $t=0.6$ with $N=1024$ points,
compared with a hyperbola of the form
\begin{equation}
\frac{(x-x_0)^2}{a^2}+\frac{y^2}{b^2}=1, 
\quad a=0.532,\ \ b=.199, \ \ x_0=2.398.
\end{equation}
This hyperbola was found using the polyfit function in julia by fitting 150
values of $Y^2$ to a quadratic function of $X$. 
\begin{figure}
\includegraphics[height=3.5in]
{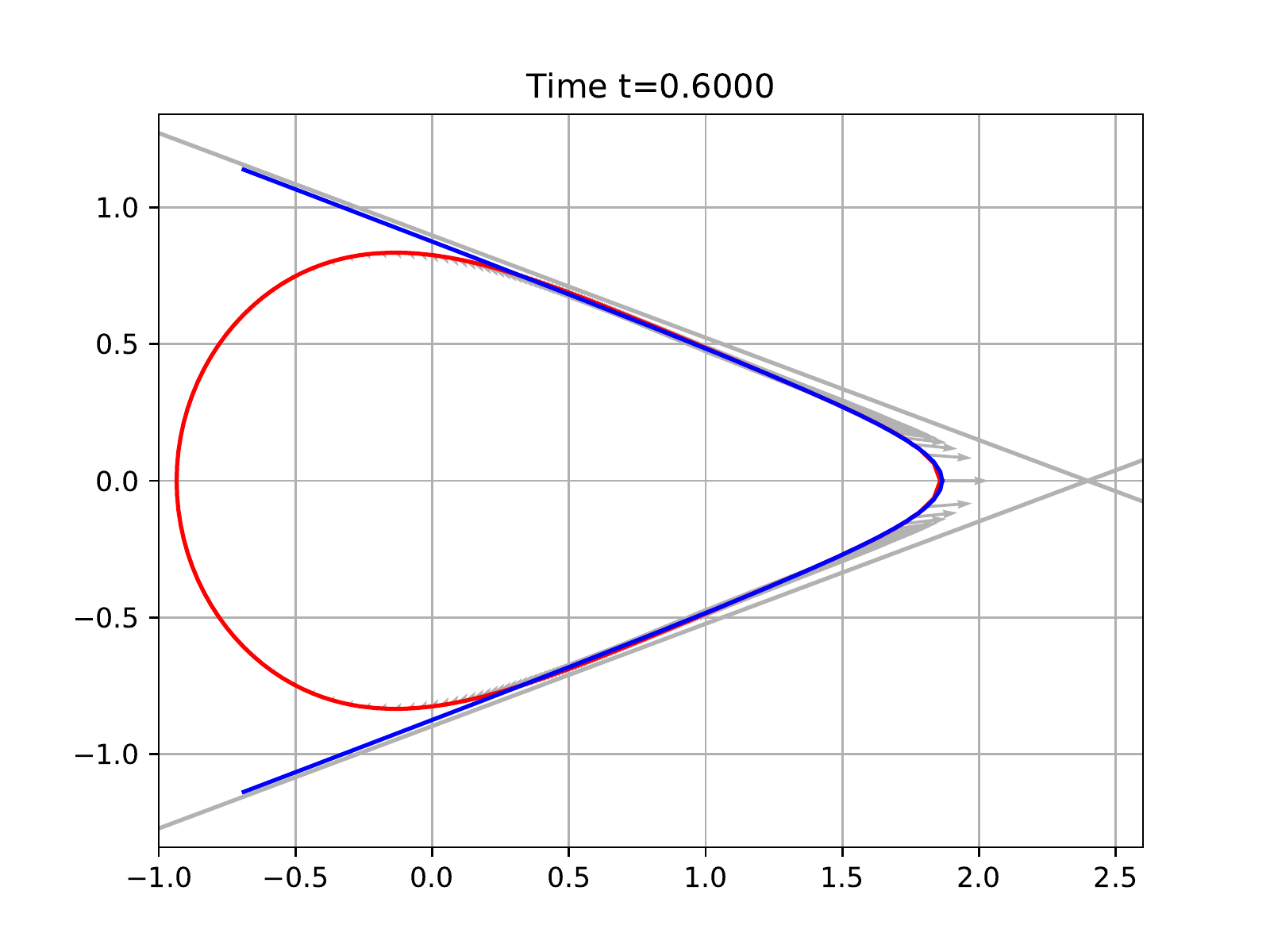}
\caption{Single-mode initial velocity with fit to hyperbola}
\label{fig1mod}
\end{figure}
%

The excellent fit of the hyperbola to the ``Pinocchio-like nose''
developing from the fluid domain suggests that no singularity
will ever form as time increases. Rather the nose should grow 
without bound, with decreasing angle between the asymptotes
of the hyperbola, like Longuet-Higgens' exact Dirichlet hyperbola solutions
that we described in section~\ref{s:dirichlet}.

\subsection{A scenario for corner formation}
\label{ss:corners}

\subsubsection{Initial data.} As discussed earlier, we seek to approximate a smooth flow expanding
away from a sharp corner. To do so, we will specify an initial interface that is a 
smooth approximation to a wedge-shaped domain $\Omega_\Theta$ 
with opening angle $\Theta\in(0,\pi)$ (as measured outside the fluid domain).
The wedge domain can be parametrized by the unit disk $\bbD$ by composing
the map defined by 
\begin{equation}
\zeta_\Theta(w):= w^\pow, \qquad \pow = 2-\frac{\Theta}{\pi}\in(1,2),
\end{equation}
that takes the right half plane $\Re w>0$ onto $\Omega_\Theta$, with a map
\begin{equation}
\zeta_+(w) = C_+\left(-1 + \frac{2}{1-w}\right), \quad C_+>0,
\end{equation}
that takes the unit disk onto the right half plane.

To fashion a smooth, bounded approximation to this infinite, singular domain,
we use a map that takes the unit disk slightly inside itself, 
creating a small `dimple' (with an approximately Gaussian shape)
near the point $w=-1$, according to the prescription
\[
\zeta_d(w ) = 
w \exp\left( -(I+iH)
\left(\eps_1\cos^{\pow_1}\frac\vartheta2+\eps_2\cos^{\pow_2}\vartheta\right) \right ),
\quad \vartheta = \arg(-w),
\]
where we take $\eps_1=0.1$, $\pow_1= 81$, $\eps_2=10^{-5}$, $\pow_2=20\pow_1$.
Finally, the initial interface at time $t_0=1$ is determined by the composition
\begin{equation}
\bm Z(w,t_0) = \zeta_\Theta\circ \zeta_+\circ \zeta_d\circ \zeta_r(w),
\end{equation}
where the first map $\zeta_r$ is the M\"obius automorphism of the disk $\bbD$
from \eqref{e:nosedata} and is used to concentrate points on one side of the circle,
and we take $C_+=2/\eps_1$ which results in $\bm Z(-1,t_0)\approx 1$.

For the wedge domain, a holomorphic velocity potential 
$f_\Theta(z) = z^{\alpha}$ determines a velocity field $(u,v)$ via
\begin{equation}\label{e:zalpha}
u-iv = {f'_\Theta(z)} = \alpha z^{\alpha-1} \,. 
\end{equation}
We use this velocity formula to determine initial data for the parametrized domain via
\[
\bar{\bm U}(w,t_0) = f'_\Theta\circ\bm Z(w,t_0).
\]
On the wedge boundary where $z\in\D\Omega_\Theta$, we have 
$\arg z = \frac\pi2 \pow$ and $\arg(u+iv)=(1-\alpha)\arg z$.
If we make the choice (as was always done in \cite{LPmajda})
\begin{equation} \label{choice:a}
\alpha=\frac1\pow \,,
\end{equation}
then
$\alpha\arg z = \frac\pi2$ and the velocity is normal to the wedge boundary.
Below we exhibit examples both with and without this choice, 
corresponding to the two cases $\alpha \pow = 1$ and $\alpha \pow = \frac34$.

\begin{figure}
\phantom{h}\hspace{-0.3cm}
\includegraphics
[height=3.0in]
{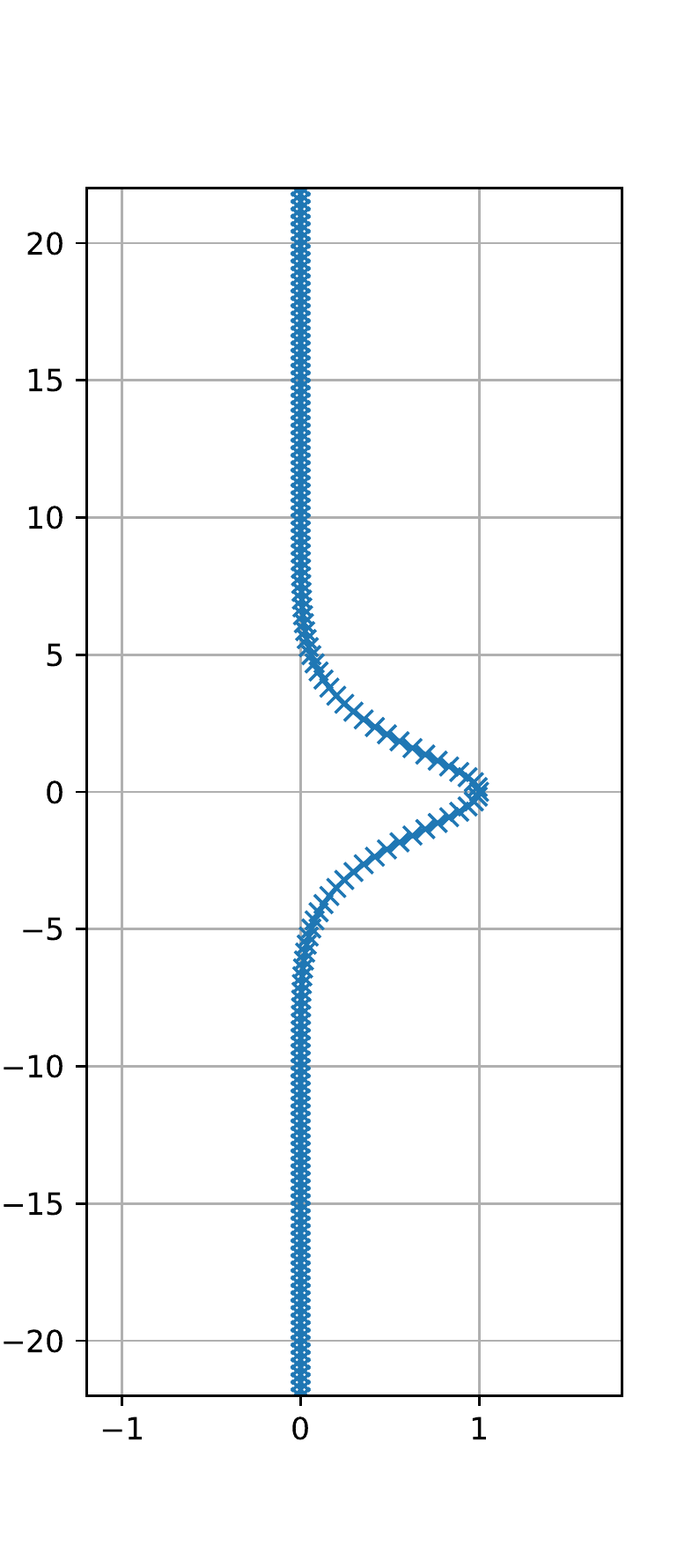}
\qquad
\includegraphics
[height=3.0in]
{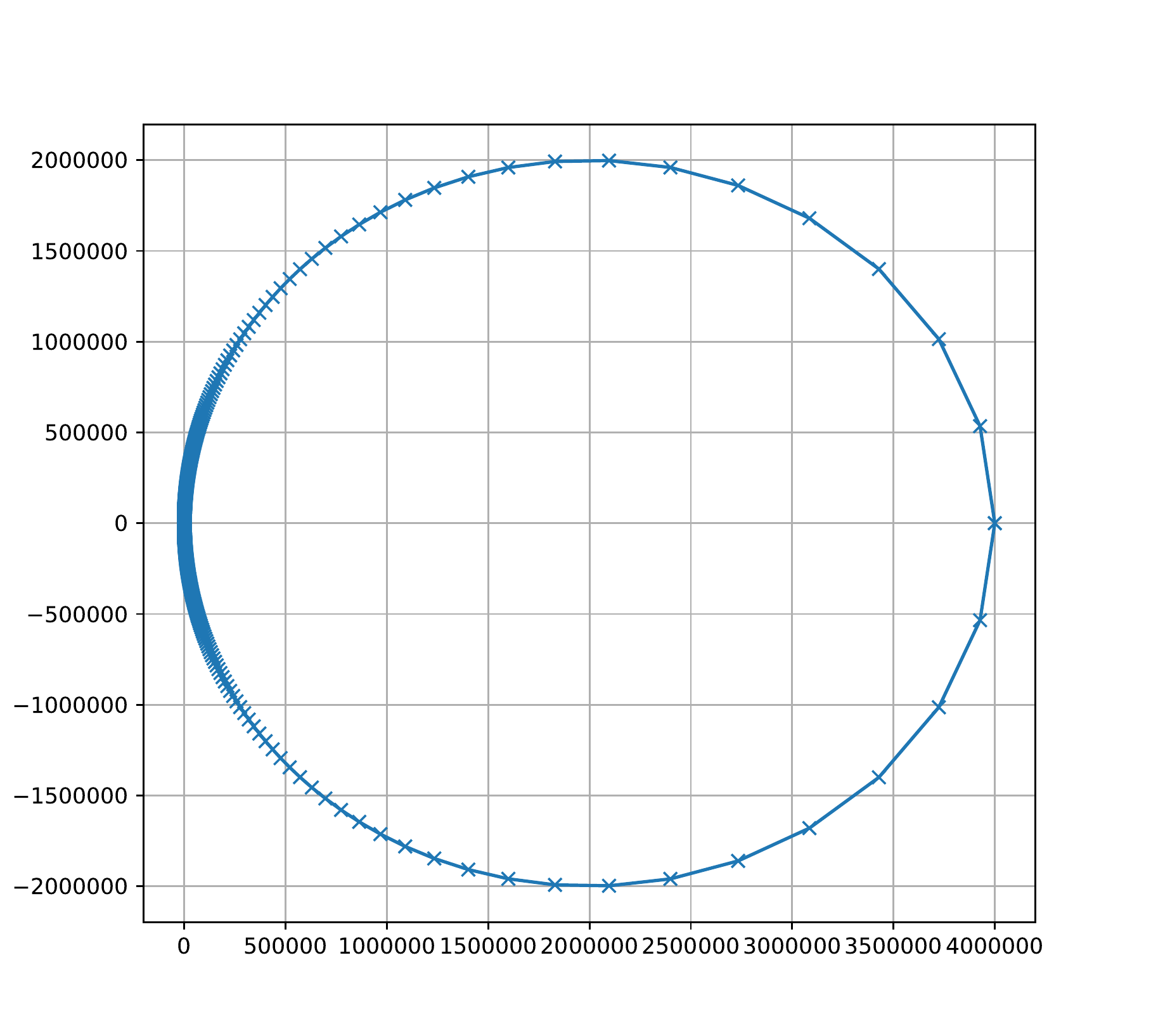}
\caption{The interface $Z(\th,t_0)$ for $\pow=1$, $N=2^{15}$ }
\label{figdimp}
\end{figure}

{\em Parameters.} 
For all the computations reported here, we discretize $\theta\in[0,2\pi)$ using $N=2^{15}$ points. 
We take the grid compression ratio in \eqref{d:compress} to be 
$c=20000$ in the case $\alpha \pow=1$,
and $c=4000$ in the case $\alpha \pow=\frac34$, and determine $r$ accordingly.

In Fig.~\ref{figdimp} we illustrate the effect of the regularizing map $\zeta_d$
by taking $\pow=1$ and plotting $Z(\th,t_0)=\bm Z(e^{i\th},t_0)$.
The left panel indicates there is a large region around $0$ where the interface is 
very close to flat, aside from a roughly Gaussian-shaped depression.
The right panel shows how the behavior of the interface near infinity is regularized 
by the factor in $\zeta_d$ with $\eps_2$, yielding $\bm Z(1,t_0)\approx 4/(\eps_1\eps_2)=4\times10^6$.

In Fig.~\ref{f:IC60deg} we plot $-X$ vs $Y$ and corresponding velocity 
for the initial interface in the case $\Theta=60^\circ$, $\pow=\frac53$, $\alpha \pow=\frac34$.
In this orientation, `water' is below `air' in the zoomed-in left panel. 
The arrows inidicate the initial velocity in the case $\alpha \pow=1$,
but we use the same initial interface also in the case $\alpha \pow=\frac34$.

\begin{figure}
\phantom{h}\hspace{-0.3cm}
\includegraphics
[width=3.0in]
{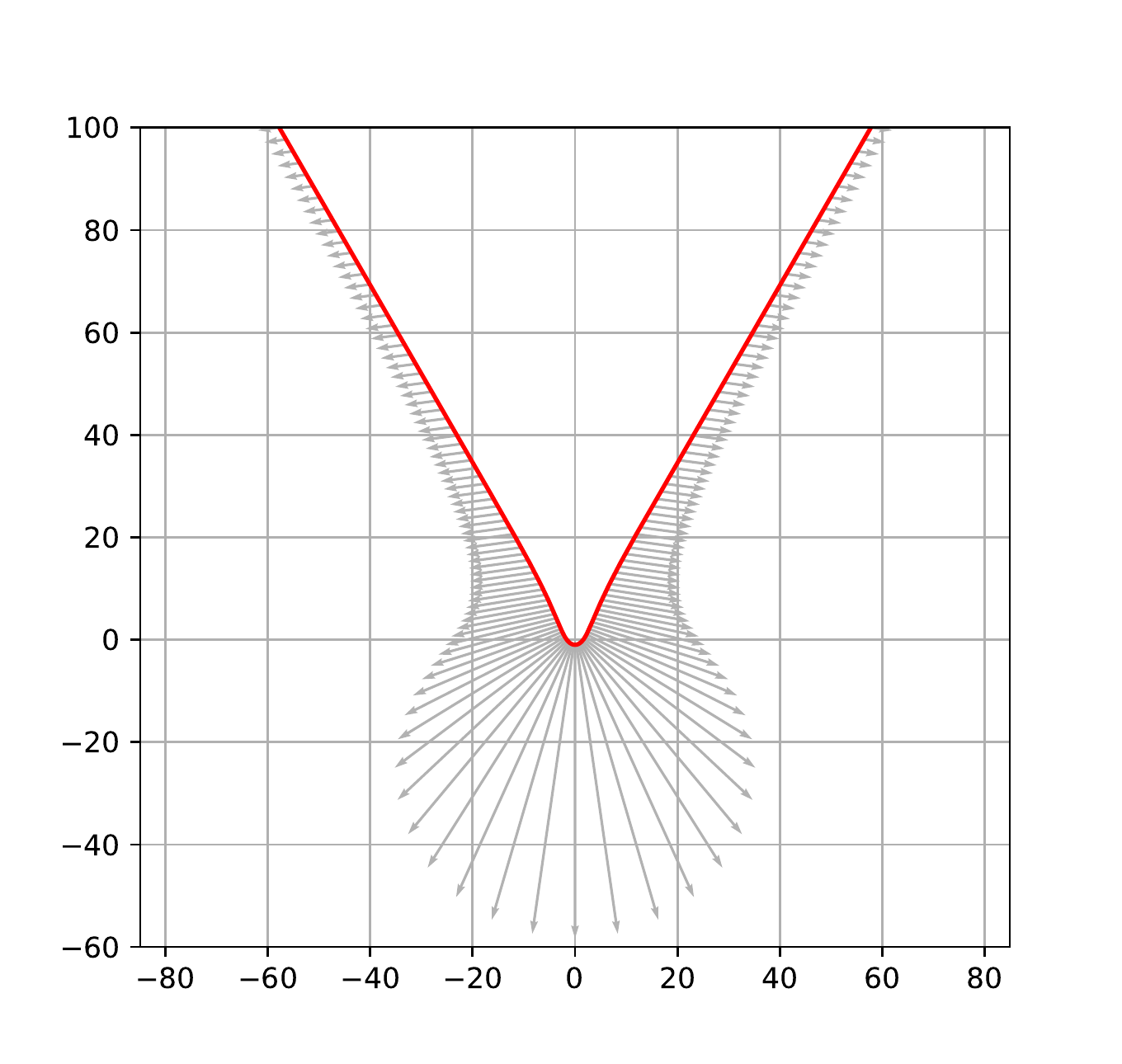}
\includegraphics
[width=3.0in]
{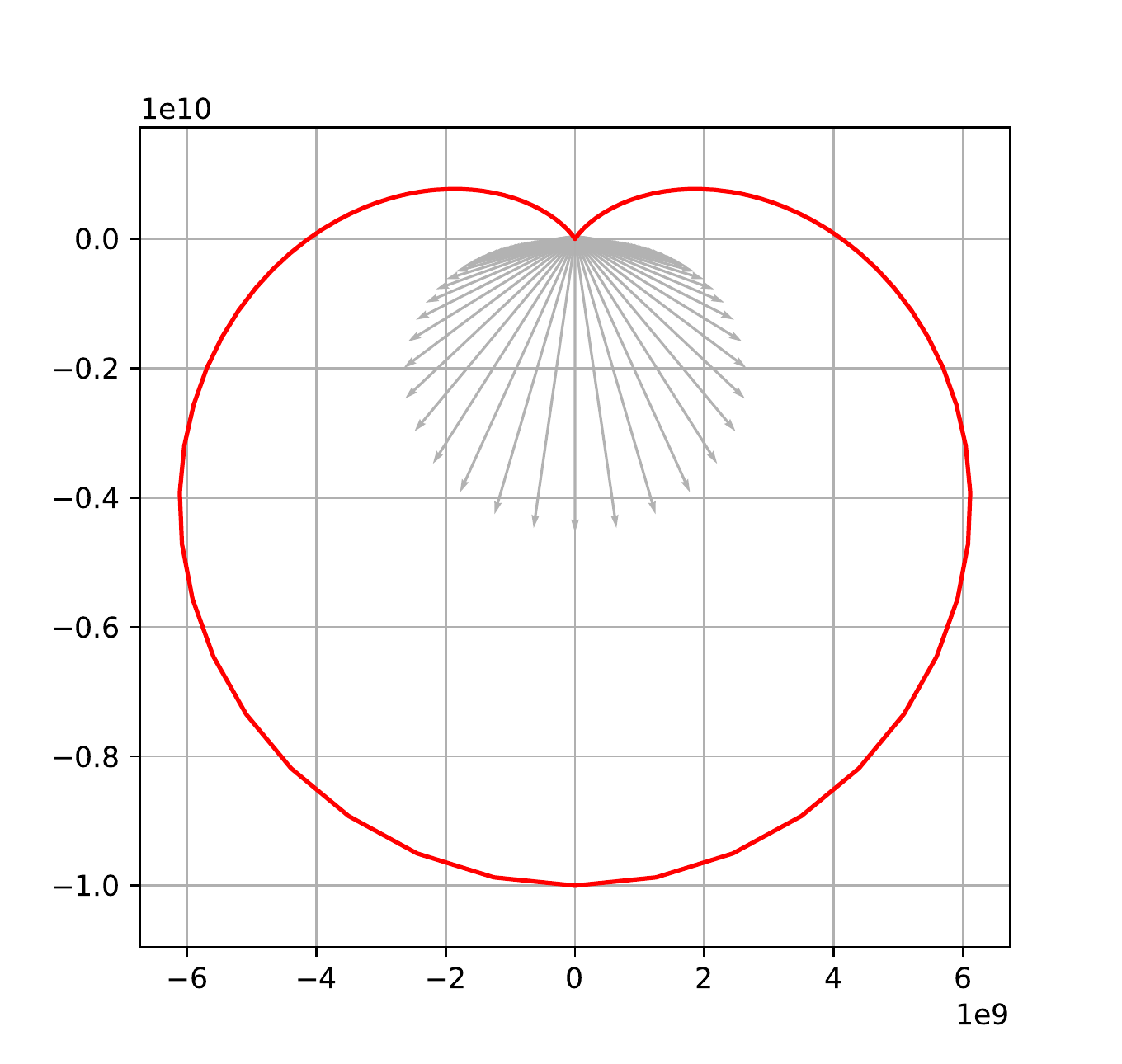}
\caption{Initial interface for $\Theta=60^\circ$ 
and velocity for $\alpha =\frac9{20}$ in `water' orientation}
\label{f:IC60deg}
\end{figure}

\subsubsection{Time evolution.} The computations reported here are carried out using 
the (less singular) variables $\bm Q=1/\bm Z_w$, $\bm V=\bar{\bm U}$ and 
$\bm S=\bm Z^{-1/\nu}$ satisfying equations \eqref{e:QU} and \eqref{e:S}. 
The results are generally consistent with those reported in 
\cite{LPmajda} which were performed using equations \eqref{e:ZF} on somewhat less singular domains.
E.g., Fig.~\ref{f:IC60deg} indicates that at the initial time $t_0=1$,
$\bm Z(1,t_0)\approx 10^{10}$ whereas this was $\approx 10^3$ for the case
$\Theta=90^\circ$ considered in \cite{LPmajda}.

\nwc{\cwidth}{4.5in}
\begin{figure}[ht]
\includegraphics[width=\cwidth]
{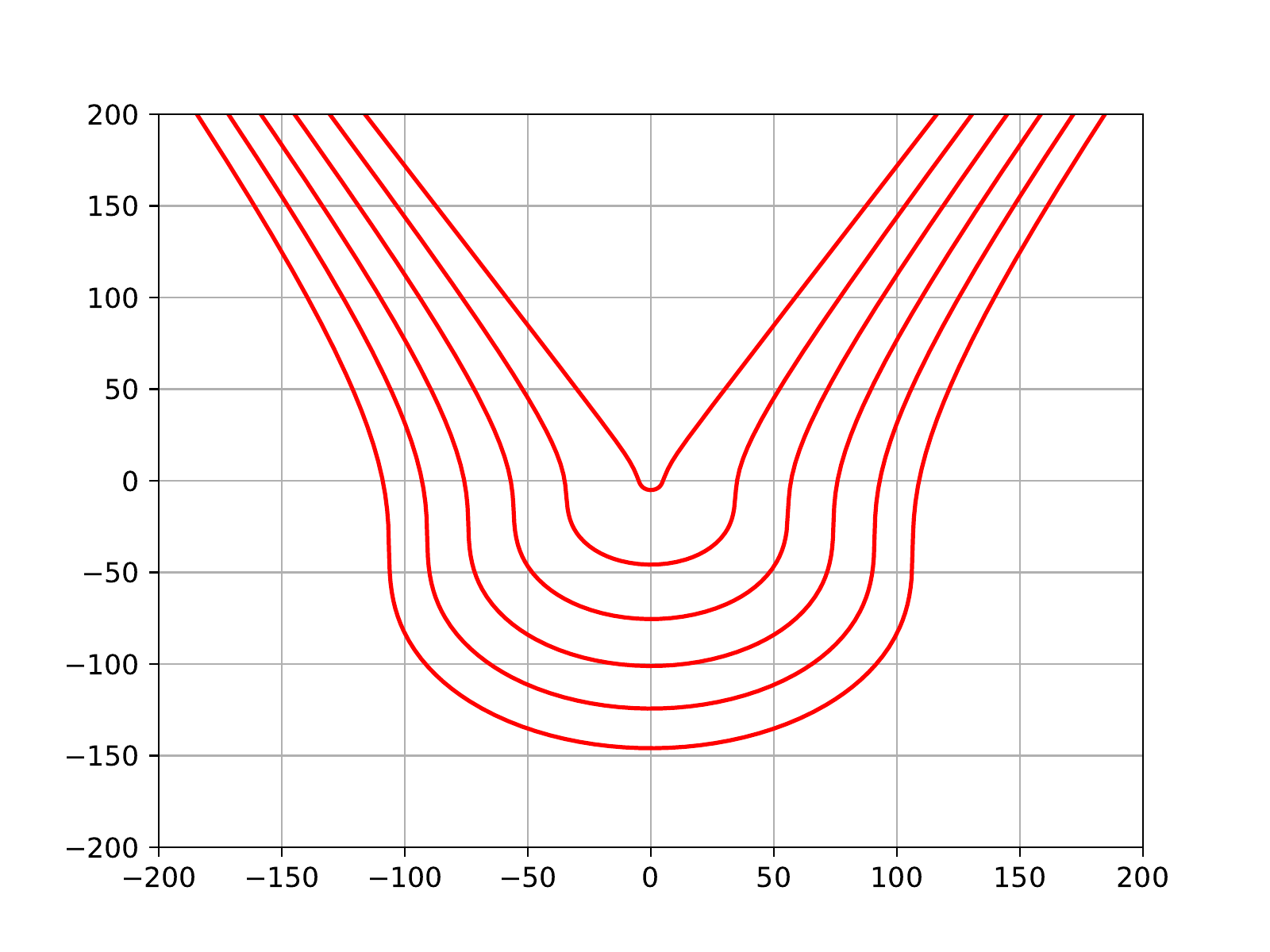}
\put(-161,7){\large $y$}
\put(-320,121){\large $-x$}
\caption{Interface for $t=10$ and $200n$ for $1\le n\le5$, with $\Theta=60^\circ$, $\alpha =\frac35$. }
\label{f:60deg}
\end{figure}

\begin{figure}[ht]
\includegraphics[width=\cwidth]
{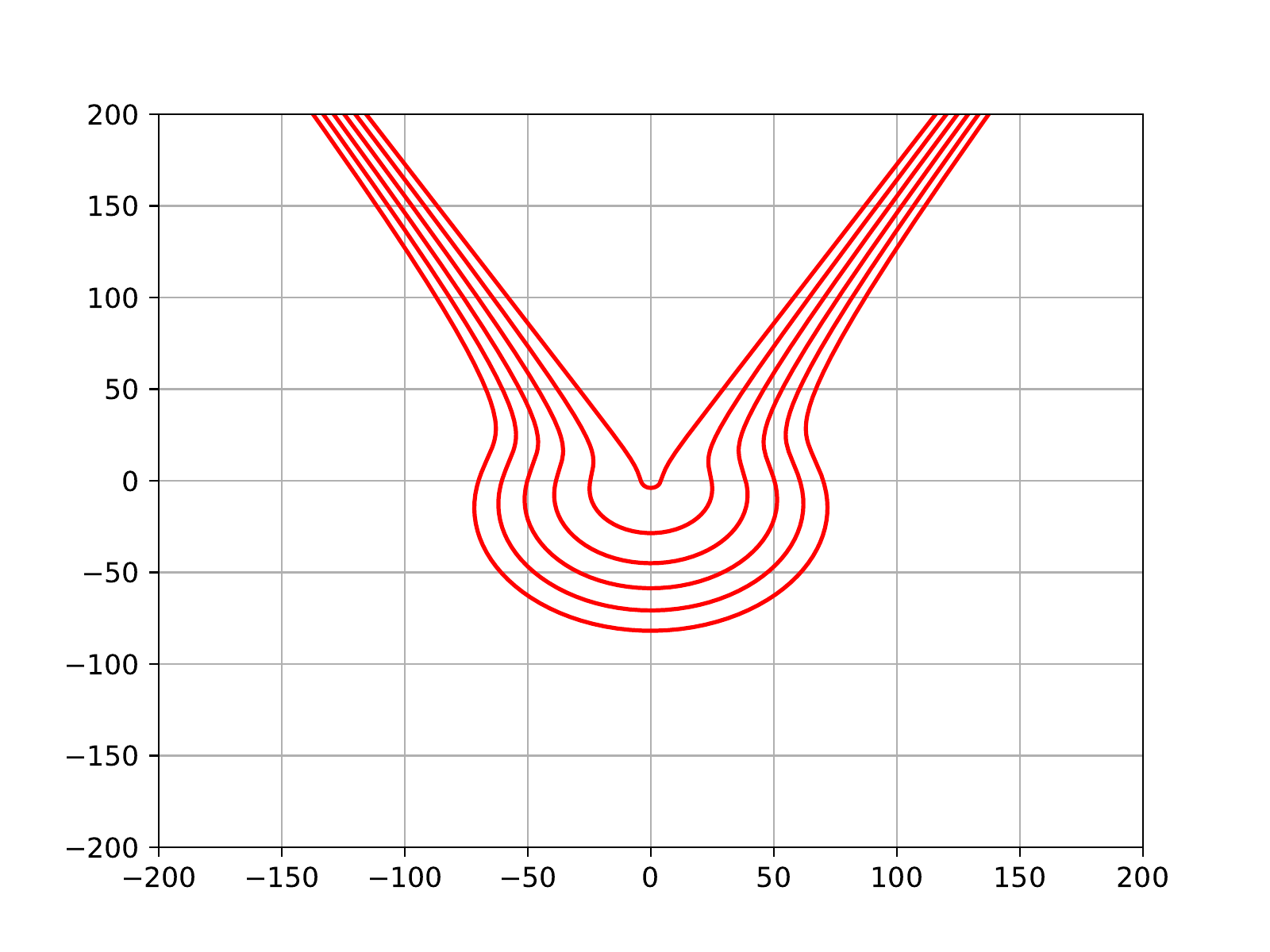}
\put(-161,7){\large $y$}
\put(-320,121){\large $-x$}
\caption{Interface for $t=10$ and $200n$ for $1\le n\le5$, with $\Theta=60^\circ$, $\alpha =\frac9{20}$. }
\label{f:60deg075}
\end{figure}

As shown in Fig.~\ref{f:60deg} (for $\alpha \pow=1$) 
and in Fig.~\ref{f:60deg075} (for $\alpha \pow = 0.75$), 
the interface expands away from the origin and the curvature decreases by a large factor.
Upon rescaling as described with small $\eps$ as described above, this is consistent
with the possible development of an interface with an initial tiny radius of curvature 
to one with radius of curvature hundreds or thousands of times larger.

\subsubsection{Evidence for self-similarity.}
We emphasize here that the Euler equations here are invariant under scaling time and space by the same factor.
Thus, although we take $\bm Z(-1,1)\approx 1$ in our computations, this corresponds to 
another solution with $\bm Z(-1,\eps)\approx \eps$ for any $\eps>0$.

The numerical results above and in \cite{LPmajda} lead us to expect
a power-law scaling in time with $\bm Z(-1,t)\sim ct^\beta$ for some 
$c,\beta>0$.  
In Figs.~\ref{f:scaledZinv} 
and \ref{f:scaledZinv075} 
we plot the {\em reciprocals} of 
time-scaled interfaces in the water orientation, 
for the choice (explained below)
\begin{equation} \label{e:betaalpha}
\beta = \frac 1{2-\alpha}\,.
\end{equation}
For $\alpha=\frac1\pow=\frac35$ this yields $\beta = \frac57$,
and for $\alpha = \frac{3}{4\pow}$ we get $\beta = \frac{20}{31}.$
Precisely, we plot 
\begin{equation}\label{e:scaleZinv}
\frac{-i t^\beta}{Z(\theta,t)}
\end{equation}
at a sequence of times $t=t_n$ for $n=1,2,\ldots 10$ 
($t_n = 500n$ for $\alpha \pow=1$, $t_n=1000n$ for $\alpha \pow=\frac34$).
These plots allows one to visualize the entire inverted fluid domain, 
with the origin of the plot corresponding to the fluid far field.

\begin{figure}
\includegraphics[width=\cwidth]
{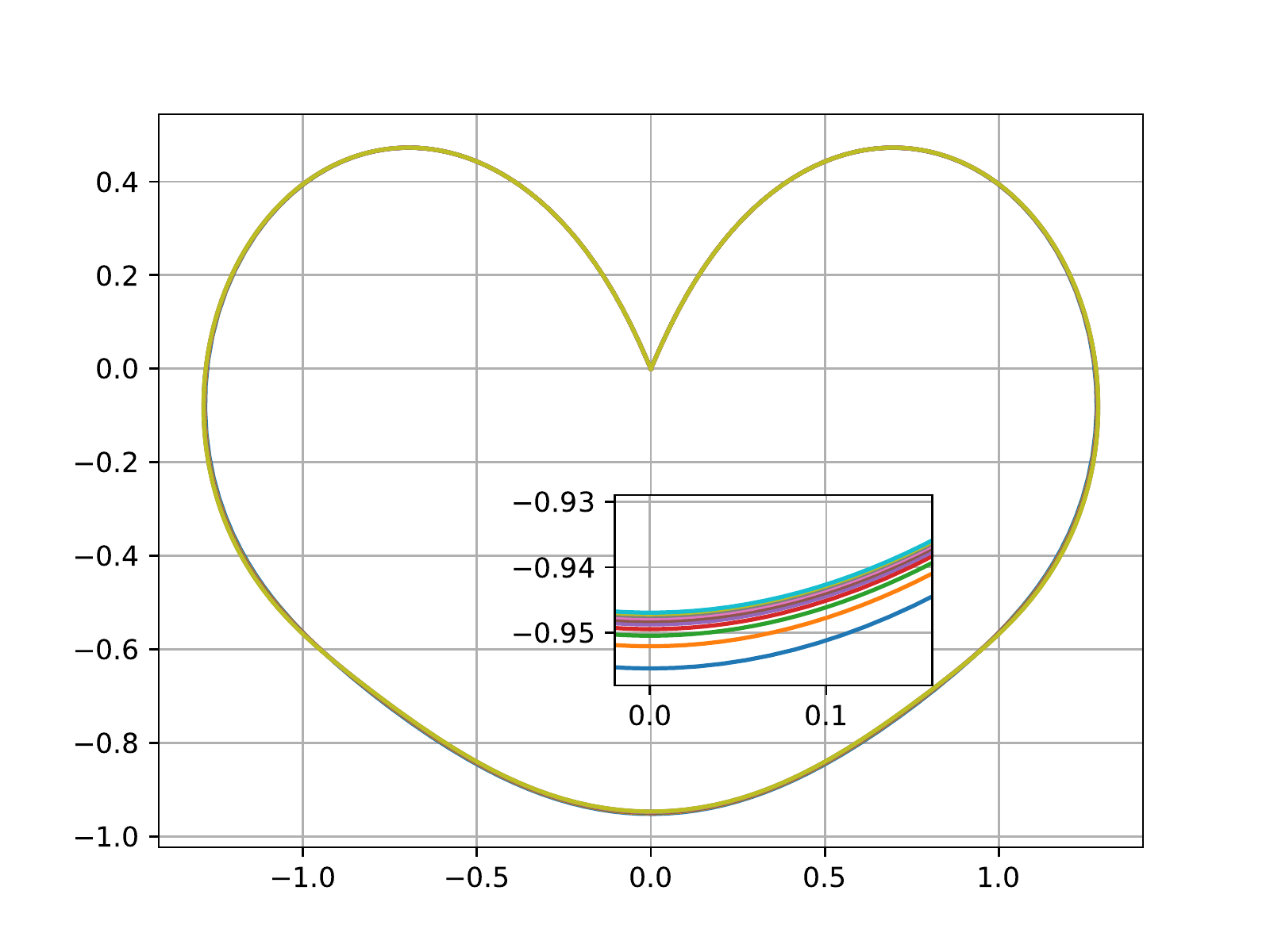}
\caption{Scaled inverse interfaces at $t= 500n$, $1\le n\le10$, $\Theta=60^\circ$,
$\alpha =\frac35$}
\label{f:scaledZinv}
\end{figure}

\begin{figure}
\includegraphics[width=\cwidth]
{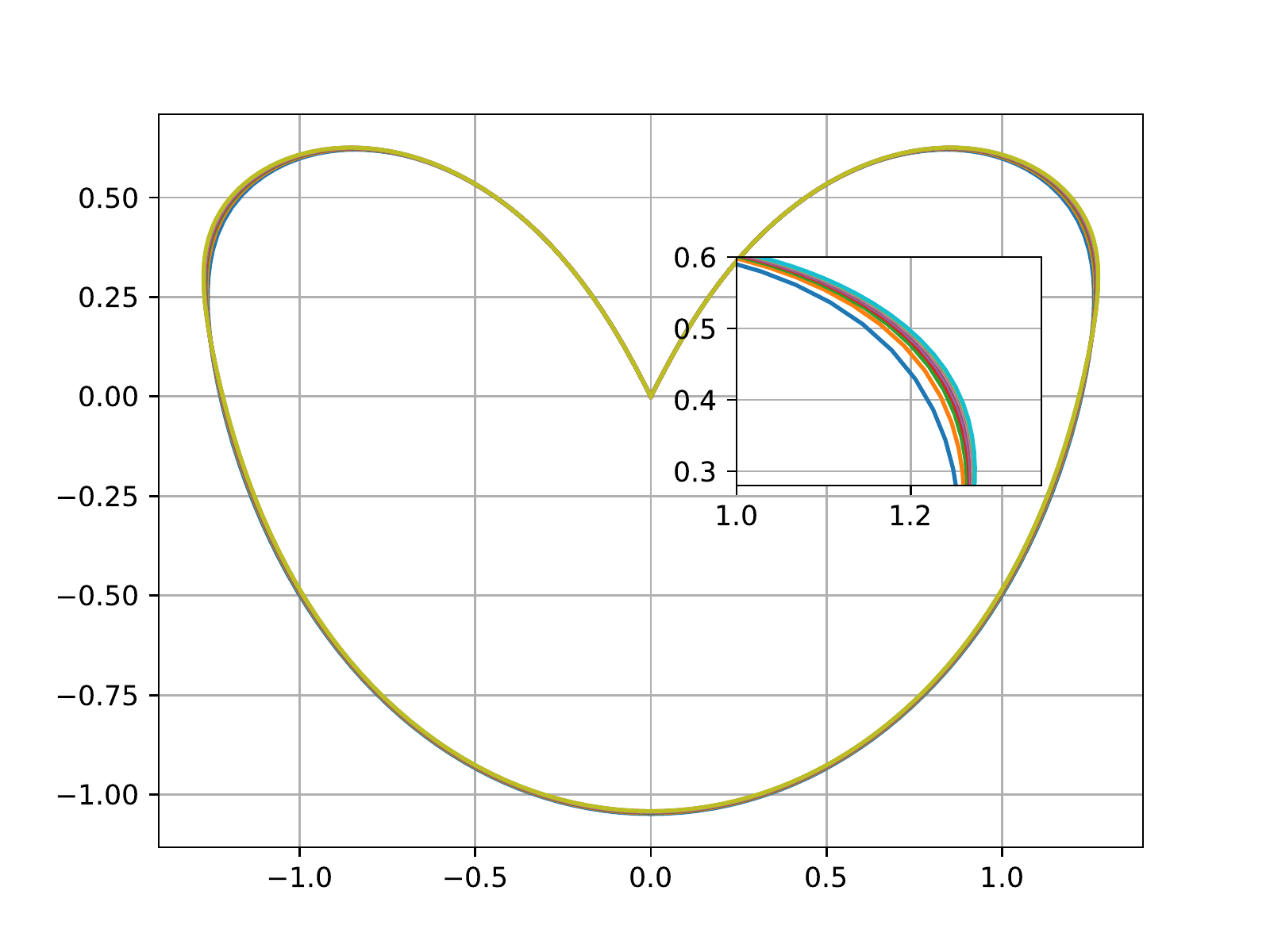}
\caption{Scaled inverse interfaces at $t= 1000n$, $1\le n\le10$, $\Theta=60^\circ$,
$\alpha =\frac9{20}$}
\label{f:scaledZinv075}
\end{figure}

The plots demonstrate a very tight collapse, with the zoomed-in view 
suggesting convergence to an invariant limit shape. 
By the scaling invariance argument mentioned above, 
this result suggests the possible existence of a 
self-similar solution starting at time $t=0$ from the exact wedge 
$\Omega_\Theta$ with initial power-law potential $f_\Theta(z)=z^\alpha$.

\subsubsection{Scaling argument}
We can provide a two-step heuristic argument 
to explain the scaling exponent $\beta$ in \eqref{e:betaalpha}.
Motivated by the results above, it is natural to seek a self-similar solution
to the governing equations \eqref{1.laplace}--\eqref{1.pzero} by scaling
the space variables and potential according to 
\begin{equation}
(x,y) = (t^\beta \tilde x, t^\beta\tilde y), \qquad
\phi(x,y,t) = t^\gamma\tilde\phi(\tilde x,\tilde y)\,. 
\end{equation}
Simple substitution shows that the linear terms in the 
Bernoulli equation \eqref{1.bernoulli} can balance 
with the nonlinear terms on the boundary only if
\begin{equation} \label{e:gbrel}
\gamma = 2\beta -1\,. 
\end{equation}
This is the first step. In the second step, writing
$\tilde\phi = \re\tilde f(\tilde z)$ with $\tilde z=\tilde x+i\tilde y$,
we find that the initial data requirement in \eqref{e:zalpha}
as $t\downarrow0$ imposes the limit relation
\begin{equation}\label{e:ulimit}
u - iv = t^{\gamma - \beta}\tilde f'(t^{-\beta} z) \to \alpha z^{\alpha-1}
\end{equation}
when taking $t\downarrow0$ with $z=x+iy$ fixed.
Thus we can expect that 
\[
\tilde f'(\tilde z) \sim \alpha \tilde z^{\alpha-1}
\quad\mbox{as $\tilde z=t^{-\beta}z\to\infty$},
\]
which entails the relation
\begin{equation}\label{e:gabrel} 
\gamma = \alpha\beta . 
\end{equation}
Putting \eqref{e:gabrel} together with \eqref{e:gbrel} yields \eqref{e:betaalpha}.

\section{Discussion}

Our experience in computing protruberances that emerge from the main body
of fluid is consistent with the observations and suggestions of 
Longuet-Higgens, who pointed out that such protrusions may often
take the form of Dirichlet hyperboloids. 
Such hyperboloidal jets exist with narrowing ``nose'' that persist and 
remain smooth forever, despite lacking any regularizing 
effect from surface tension, viscosity, or gravity. 
Thus ideal droplets with smooth boundary do not seem to readily 
form local singularities, consistent with the analytical constraints 
provided in the work of Kinsey and Wu~\cite{KinseyWu2018} and Wu~\cite{wu2015blow}.

However, our computations presented in section \ref{ss:corners} 
strongly suggest that corner formation may occur in an unstable manner
for specially prepared initial data in bounded domains.
After time-reversal, it appears that a smooth interface may emerge from 
initial data containing a corner.

``Zooming in,'' it appears plausible that a corner can form in
an asymptotically self-similar way, with a rather general exterior angle 
and power-law velocity profile approaching the corner.
Thus we conjecture that a two-parameter family of 
self-similar solutions of the ideal droplet equations may exist
in infinite, asymptotically wedge-shaped domains. 
Determining whether such solutions do indeed exist appears to be a difficult
challenge for analysis.
We point out that if such solutions exist with $x$-intercept proportional
to $t^\beta$, then the acceleration of the interface at this point
would blow up strongly, like $t^{\beta-2}$.

\section*{Acknowledgements}
The authors are grateful to Sergey Gavriluk for historical references 
regarding ellipsoidal solutions.
This material is based upon work supported by 
the National Science Foundation under 
grants DMS 1812673 and DMS 2106988 (JGL)
and grants DMS 1812609 and DMS 2106534 (RLP).


\appendix

\section{Conformal mapping formulation and least action principles}\label{sec:setup}

For water wave equations with gravity,
conformal mapping formulations of time dependent flows were
described and employed for purposes of analysis by, for example,
Ovsjannikov \cite{Ovs71}, Kano \& Nishida \cite{KN79},
and Wu \cite{Wu97},
and recently by Hunter et al. \cite{HIT16}.
For purposes of numerics and formal analysis such conformal formulations
were described by Dyachenko {\it et al.} \cite{dzk96}, 
Chalikov and Sheinin \cite{Chalikov96},
Choi and Camassa \cite{Choi99}, and were extensively developed recently 
in \cite{Dyachenko2016,DyachenkoNewell2016,Lushnikov2017}.

Our aim in this appendix is to provide a pedagogical derivation of
the ideal droplet equations in terms of a time-dependent conformal mapping of the
disk $\bbD$ onto the fluid domain $\Omega_t$, 
especially taking into account the freedom of choice in parametrization 
afforded by the automorphism group of $\bbD$.
Using a special, simple choice that fixes one point of the map from $\bbD$ to $\Omega_t$, 
a conformal mapping formulation for droplets was written already by
Ovsjannikov \cite{Ovsjannikov71} in complex form with zero surface
tension and body force. 

The resulting conformal ideal droplet equations
form a type of water wave system with one spatial dimension, 
governing the time evolution of the boundary parametrization and the 
velocity potential at the boundary.  The equations are nonlocal and nonlinear,
but all nonlocality is expressed in terms of the $2\pi$-periodic 
Hilbert transform $H$ on the real line, i.e., on $\bbT=\R/2\pi\Z$.

\subsection{Ideal droplet equations}
We consider a two-dimensional drop of fluid that at time $t$ occupies a domain $\Omega_t$
conformally equivalent to the unit disk $\bbD$ in the complex plane. 
We suppose the fluid is inviscid and the motion incompressible. 
Let $u=(u_1, u_2)$ denote the velocity field and $p$ the pressure. 
We assume that the external body force $-\nabla \bodypot$ is
conservative, arising from a potential $\bodypot$. 
Then inside $\Omega_t$ we have the Euler equations
\begin{gather}\label{eq:incompeuler}
u_t+u\cdot\nabla u+\nabla p=-\nabla \bodypot,\ \ \ \ \nabla\cdot u=0.
\end{gather}
Assuming the flow is irrotational, the velocity must be the gradient of a potential $\phi$,
satisfying $u=\nabla\phi$, with $\phi$ harmonic in $\Omega_t$. 
As is usual, because 
$u\cdot\nabla u=\frac{1}{2}\nabla|\nabla\phi|^2$, 
upon integration the Euler equations \eqref{eq:incompeuler} yield the Bernoulli equation 
\begin{equation}\label{e:Bernoulli1}
    \phi_t+\frac{1}{2}|\nabla\phi|^2+p+\bodypot=c(t) 
\end{equation}
in $\Omega_t$, where $c(t)$ is a real constant in space.
The value of $p$ on the boundary $\partial\Omega_t$ is determined by the surface tension
\begin{gather}\label{e:pkappa}
p=\gamma\kappa,
\end{gather}
where $\gamma$ is the surface tension coefficient and $\kappa$ is the (mean)
curvature of $\partial\Omega_t$, chosen so the curvature of a circle is positive.
The equations of motion are completed by specifying a kinematic equation which
states that the normal velocity of the fluid domain $\Omega_t$ agrees with
the normal component of fluid velocity $\nabla\phi$.

\subsection{Conformal mapping formulation}

\subsubsection{Notations and basic notions.}
We let $(x,y)$ denote Eulerian coordinates inside the 
fluid drop $\Omega_t$ and write $z=x+iy$.
The function $f=\phi+i\psi$ denotes the (holomorphic) 
complex velocity potential, so the Cauchy-Riemann equations
$\phi_x=\psi_y$, $\phi_y=-\psi_x$ relate the stream function $\psi$
to $\phi$.  The complex fluid velocity then satisfies
\begin{equation}\label{e:fz}
    u = u_1+iu_2 = \bar f_z.
\end{equation}
We parametrize the fluid domain $\Omega_t$ at time $t$ by a holomorphic
bijection 
 $\mb Z(\cdot,t)\colon \bbD\to \Omega_t$ denoted
\[
    z=\mb Z(w,t) = \mb X(w,t)+ i\mb Y(w,t) \,, \qquad w\in \bbD\,.
\]
In terms of the conformal variable $w$, the (holomorphic) velocity
potential and (anti-holomorphic) velocity are written
\[
    \mb F(w,t) = \mb\Phi(w,t)+i\mb\Psi(w,t) = f\circ\mb Z \,, \qquad
    \mb U(w,t)=u\circ\mb Z\,. 
\]
Here the notation $f\circ\mb Z$ indicates $f(\mb Z(w,t),t)$, etc.
In these terms, the relation \eqref{e:fz} is expressed as
\begin{equation}\label{d:barU}
    \bar{\mb U} = \frac{\mb F_w}{\mb Z_w}\,.
\end{equation}
On the boundary $S^1=\partial \bbD$ of the disk we write $w=e^{i\th}$ with 
$\th\in\bbT={\mathbb R}/{2\pi\mathbb Z}$.  The key dependent variables
in the conformal formulation are the traces of $\mb Z$ and $\mb F$ on $\partial\bbD$,
and we write
\begin{equation}
    Z(\th,t)=X(\th,t)+iY(\th,t) = \mb Z(e^{i\th},t)\,, \qquad
    F(\th,t)=\Phi(\th,t)+i\Psi(\th,t) = \mb F(e^{i\th},t)\,. \qquad
\end{equation}

\subsubsection{Fourier and Hilbert transforms.}
We will write the Fourier expansion of a general square-integrable function $G\colon\bbT\to \C$ as 
\[
    G(\th) = \sum_{k\in\Z} \hat G_k e^{ik\th} \,, 
    \qquad \hat G_k = \frac1{2\pi}\int_\bbT G(\th)e^{-ik\th}\,d\th\,.
\]
Such a function has a holomorphic extension $\mb G$ inside the disk $\bbD$ if and only if all
negative modes vanish, meaning $\hat G_k=0$ for all $k<0$.
If we write 
\[
G=\hat G_0+A+iB
\] where $A$ and $B$ are real, then this means
$B=HA$ and $HB=-A$, where $H$ is the Hilbert transform on $\bbT$, satisfying
\[
HG(\theta) = \sum_{k\ne0} (-i\sgn k)\hat G_k e^{ik\th} = 
{\mathop{\rm p.v.}} 
\int_{\bbT}  G(s)\cot\left(\frac{\th-s}2\right)\frac{ds}{2\pi}\,.
\]
(Note however that $-H^2$ is not the identity on $\bbT$, unlike the Hilbert transform on $\R$.)
Thus $G$ has an extension $\mb G$ holomorphic in $\bbD$ exactly when 
$\hat B_k=(-i\sgn k)\hat A_k$ for all $k\ne0$,
and in this case 
\begin{equation}\label{e:ABk}
\mb G(w) = \sum_{k\ge0} \hat G_k w^k \quad\mbox{ for all $w\in\bbD$,}
\qquad
\hat G_k=\hat A_k+i\hat B_k=2\hat A_k \quad\mbox{ for all $k>0$. }
\end{equation}

\subsubsection{Choice of conformal map I} 
Because of the well-known characterization of holomorphic automorphisms of the disk $\bbD$ as
having the form
\[
g(w) = e^{ia} \frac{\alpha-w}{1-\bar\alpha w},\qquad a\in\bbT, \ \ \alpha\in\bbD,
\]
there is some freedom in the choice of the conformal map $\mb Z(\cdot,t)\colon\bbD\to\Omega_t$. 
By a well-known version of the Riemann mapping theorem \cite[Theorem 3.3.1]{ss10}, 
$\mb Z$ will be uniquely determined once we specify the image of the origin
\begin{equation}\label{d:Z0}
\mb Z_0(t):=\mb Z(0,t)=\hat Z_0(t)
\end{equation} 
(within the fluid domain), 
and the argument of the derivative $\mb Z_w(0,t)$, which we denote by $\th_1(t)$.
Using the Cauchy integral formula we calculate
\begin{equation}\label{d:beta}
    \mb Z_w(0,t) = \frac1{2\pi i}\int_{S^1} \frac{\mb Z}{w^2}\,dw = 
\frac1{2\pi} \int_{\bbT} Z e^{-i\th}\,d\th =\hat Z_1(t)\,,
\end{equation}
and observe that $\th_1(t)$ evolves according to 
\begin{equation}\label{e:argt}
\th_1'(t) 
    = \frac{d}{dt} \arg \hat Z_1(t) =\Im  \frac{\hat Z_1'}{\hat Z_1}\,.
\end{equation}

\medskip
\subsubsection{Holomorphic kinematic equation.}\label{sss:Hkinematic}
The kinematic condition that the conformal image of $\bbD$ follows the fluid
motion corresponds to the condition that the normal component of coordinate
velocity $\mb Z_t$ matches that of the fluid velocity on $\partial\bbD$.
Thus we should explain how $\mb Z_t$ is determined by a given normal
velocity law up to conformal automorphisms of $\bbD$.

We define (non-normalized) tangent and outer normal vectors to the boundary 
of $\Omega_t$ by 
\begin{equation}
    \tau(\th,t)=Z_\th = \frac{\partial}{\partial\theta} \mb Z(e^{i\th},t)\,, \qquad
n(\th,t) =-i Z_\th\,.
\end{equation}
These functions are traces on $\partial\bbD$ of the respective holomorphic functions
\begin{equation}\label{e:holtn}
\bm\tau(w,t) = iw\mb Z_w  \,,\qquad
\bm n(w,t) = w\mb Z_w\,, \quad w\in\bbD\,.
\end{equation}
Next we introduce normal and tangential components of $Z_t$ (non-normalized) by writing 
\begin{equation}\label{e:barnZt}
\bar n Z_t = \upsilon^n+i\upsilon^\tau\,, \qquad
\upsilon^n = \Re \bar n Z_t \,, \quad 
\upsilon^\tau = \Im \bar n Z_t = \Re\bar \tau Z_t \,.
\end{equation}

Because $\mb Z_t$ is holomorphic in $\bbD$, it is essentially determined by 
the normal component $\upsilon^n$ up to automorphism, as follows. 
Because $\mb Z(\cdot,t)$ is presumed injective on $\bbD$,
the derivative $\mb Z_w$ cannot vanish inside $\bbD$. 
(For an in-depth discussion of injectivity see subsection~\ref{ss:criteria} below.)
Hence we can write
\begin{equation}\label{e:Zth}
\mb Z_t = \mb Z_w\mb G \,,
\end{equation}
where $\mb G$ is holomorphic in $\bbD$. Substituting into \eqref{e:barnZt} and using \eqref{e:holtn}, 
we infer
that on $\D\bbD$,
\begin{equation}\label{e:Grel}
     \frac Gw =
\frac{\upsilon^n}J+i\frac{\upsilon^\tau}{J}   \,,
\qquad J:= \bar n n = |Z_\th|^2\,.
\end{equation}
This expression is holomorphic in $\bbD$ except for a pole at 0 with residue $\hat G_0=\mb G(0,t)$.
Then it follows
(noting that the imaginary part of the constant term is not recovered by $iH$) 
\begin{align}
\frac{G-\hat G_0}w 
&= (I+iH)\re\left( \frac{G-\hat G_0}w \right) + i\im \hat G_1 \,.
\label{e:Gw1}
\end{align} 
Since $(I+iH)\re (\hat G_0/w)=(I+iH)\re (w\overline{\hat G_0 })=
w\overline{\hat G_0 }$ on $\D\bbD$, we infer

\begin{equation}\label{e:Gsolv1}
G = w(I+iH)\re\left(\frac{G}w\right) + \upsilon_0(\theta,t) \,,
\end{equation}
where 
\begin{align}
\upsilon_0(\theta,t) &= 
\hat G_0(t) + iw\im\hat G_1(t)- w^2\overline{\hat G_0(t)} \,.
\label{e:v0}
\end{align}
The formulas \eqref{e:Gsolv1}--\eqref{e:v0} express $G$ clearly 
as the boundary trace of a holomorphic function.
The coefficients $\hat G_0$, $\im\hat G_1$ are associated with the 
choice of conformal map as expressed in \eqref{d:Z0} and \eqref{e:argt}, via the relations
\begin{equation}\label{d:alpha}
 \hat G_0 = 
    \frac{\mb Z_{t}(0,t)}{\mb Z_w(0,t)}  =  \frac{\hat Z_0'}{ \hat Z_1} \,,
\end{equation}
\begin{equation}
    \im\hat G_1 = \im\left(\frac{\mb Z_{wt}}{\mb Z_w} - \frac{\mb Z_{ww} \mb Z_t}{\mb Z_w^2} \right)_{w=0}
    = \theta_1'(t) - \im \frac{2\hat Z_2\hat Z_0'}{\hat Z_1^{2}} \ .
\end{equation}
In this way, $\mb Z_t$ is determined by the (non-normalized) normal velocity
$\upsilon^n$ and the functions $\theta_1'$ and $\hat Z_0'=\mb Z_0'$, which may
be freely specified.

Finally, the physical kinematic relation requires that 
the normal component of $\mb Z_t$
matches that of the complex fluid velocity $\mb U$ from \eqref{d:barU}, which means simply that
\begin{equation}\label{e:match}
\upsilon^n = 
\Re \bm n \bar{\mb U} = \re w\mb F_w =
\Re(-iF_\th)=\Psi_\th  \,.
\end{equation}
This then determines $G$ as in \eqref{e:Gsolv1} with
\begin{equation}
\re\left(\frac Gw\right) = \frac{\Psi_\theta}{J}.
\end{equation}

\subsubsection{Holomorphic Bernoulli equation.} 
Because $\phi_t=\Re f_t$ and  $|\nabla\phi|^2 = |f_z|^2$,
the usual Bernoulli equation \eqref{e:Bernoulli1} 
reads
\begin{equation}\label{e:Bernoulli}
\Re f_t +\frac{1}{2}|f_z|^2 + p + \bodypot = c(t)\,.
\end{equation}
Recalling that $\mb F(w,t)=f(\mb Z(w,t),t)$, we compute
$ \mb F_t = f_t\circ \mb Z + (f_z\circ\mb Z)\mb Z_t $, so that
\[
\mb R :=
- f_t \circ \mb Z 
= -\mb F_t + \frac{\mb F_w \mb Z_t}{\mb Z_w} 
= -\mb F_t + \mb F_w \mb G \,,
\]
\[
\frac12 |f_z\circ \mb Z|^2 
=\frac12 \left|\frac{\mb F_w}{\mb Z_w}\right|^2 
    = \frac12 \frac{|F_\th|^2}{|Z_\th|^2} 
= 
\frac{\Phi_\th^2+\Psi_\th^2}{2J} 
\,.
    \]
Because $\mb R$ is holomorphic, we can use the last equation with \eqref{e:Bernoulli} 
to express the Bernoulli equation in holomorphic form as 
\begin{equation}\label{e:Ft}
    { 
    \mb F_t  = \mb F_w \mb G 
- \mb R \,, }
\qquad R = (1+iH)\left( 
\frac{\Phi_\th^2+\Psi_\th^2}{2J} 
    +p+\bodypot \right) 
+ \tilde c(t)\,,
\end{equation}
where $R(\th,t)=\mb R(e^{i\th},t)$.
Inside the fluid, the pressure in conformal variables may be expressed using
\eqref{e:Bernoulli} in the form
\begin{equation}\label{e:pcircZ}
p\circ\mb Z = \Re \mb R - \frac12\left|\frac{\mb F_w}{\mb Z_w}\right|^2 - \bodypot+c(t)\,.
\end{equation}

\subsubsection{Choice of conformal map II} 
Recall that we may freely specify $\th_1'(t)$ (real) and $\mb Z_0'$ (complex), 
as long as $\mb Z_0(t)$ does not hit $\partial\Omega_t$.
We discuss two choices of these functions that may be particularly useful.
The simplest choice is certainly to choose 
\begin{equation}\label{e:th1Z0}
    \th_1'(t) = 0,\quad \mb Z_0'(t)=0.
\end{equation}
From \eqref{e:v0} this is clearly equivalent to the condition
\begin{equation}    
    \upsilon_0(\theta,t) = 0.
\label{e:up0}
\end{equation}
This corresponds to what was done in \cite{Ovsjannikov71}, and 
can suffice for short times or symmetric flows.
In order that the point $\mb Z(0,t)$ should remain inside the fluid domain over longer times, 
however, it may be convenient to choose $\mb Z_0(t)$ to move with the fluid velocity 
$\mb U_0(t)=\mb U(0,t)$, so that  
\begin{equation}\label{e:z0t}
    \overline{\mb Z_0'(t)} = \bar{\mb U}_0(t)
= \frac{\mb F_w(0,t)}{\mb Z_w(0,t)}
= \frac{\hat F_1(t)}{\hat Z_1(t)} \,.
\end{equation}
Then it seems simplest to choose $\th_1'=2\Im \mb Z_0'\hat Z_2\hat Z_1^{-2}$, 
which makes $\im\hat G_1 = 0$. It follows
\begin{equation}\label{e:a2}
\upsilon(\th,t) = \frac{w}{|\hat Z_1|^2} \left(\overline{w \hat F_1}-w\hat F_1 \right)
= \frac{-2iw \im(w\hat F_1)}{|\hat Z_1|^2}\,.
\end{equation}

\subsubsection{Real equations.} 
Computation and analysis are carried out in terms of the real parts of the
traces $X=\re Z$ and $\Phi =\re F$ on the circle $\D\bbD$, 
relying on the Hilbert transform to recover the imaginary parts as needed.
Because $w\mb Z_w=-i\mb Z_\th = -i(X_\th+iY_\th)$ on $\D\bbD$,
from \eqref{e:Zth} we get
\begin{equation}\label{e:Xt0}
X_t = \Im\, (X_\th+i Y_\th)
\left(\frac{\upsilon^n}J+i\frac{\upsilon^\tau}{J}  \right) 
=
Y_\th \frac{\upsilon^n}J + X_\th \frac{\upsilon^\tau}J \,.
\end{equation}
Because $w\mb F_w=-iF_\th=-i(\Phi_\th+i\Psi_\th)$ on $\partial\bbD$, 
the real part of \eqref{e:Ft} yields
\begin{align}\label{e:Phit0}
\Phi_t &=
    \Im\, (\Phi_\th+i\Psi_\th)
\left(\frac{\upsilon^n}J+i\frac{\upsilon^\tau}{J}  \right)
- \frac{\Phi_\th^2+\Psi_\th^2}{2J} 
-p -\bodypot-\tilde c_1(t) 
\\ \nonumber
&=\frac1{2J}(\Psi_\th^2-\Phi_\th^2) 
+ \Phi_\th  \frac{\upsilon^\tau}J 
 - p - \bodypot - \tilde c_1(t) \,,
\end{align}
where to get the last equality we use the fact that
$\upsilon^n = \Psi_\th$ from \eqref{e:match}.
We typically choose $\tilde c_1=0$, but one could instead choose $\tilde c_1$ 
to enforce $\hat\Phi_0$ to be constant.

The pressure is given on the boundary as $p=\gamma\kappa$, 
where $\gamma$ is the coefficient of surface tension
and $\kappa$ is the curvature of the curve $\th\mapsto (X(\th,t),Y(\th,t))$. 
With the convention that the curvature of a circle is positive, the pressure at the boundary is 
\begin{gather}
p=\gamma \frac{X_{\th}Y_{\th\th}-X_{\th\th}Y_{\th}}{(X_{\th}^2+Y_{\th}^2)^{3/2}}.
\end{gather}

We summarize the equations of motion in terms of $X$ and $\Phi$
using the operator defined by 
\begin{gather}
\Lambda=H\partial_\th=\partial_\th H\,,
\end{gather} 
so that $Y_\th = HX_\th = \Lambda X$ and $\Psi_\th =H\Phi_\th=\Lambda\Phi$, e.g.
We compute that 
\begin{equation}\label{e:vtau}
\frac{\upsilon^\tau}J = \im\frac Gw = H\left(\frac{\upsilon^n}J\right) + 
\im\left(\frac{\upsilon_0}w\right),
\end{equation}
and note that
\begin{equation} \label{e:vtau1}
\frac{\upsilon_0}w = \bar w \hat G_0 - w\overline{\hat G_0} + i\Im \hat G_1  = 
i\im (2\bar w\hat G_0+\hat G_1) =: i\upsilon_1(\theta,t).
\end{equation}
Taking $\tilde c_1=0$, 
we obtain the evolution equations governing ideal droplets finally in the form
\begin{gather}
\label{e:XtLam}
X_t = (\Lambda X)\frac{\Lambda\Phi}J 
+ X_\th \left( H\left(\frac{\Lambda\Phi}J\right) 
+ \upsilon_1 \right)\,,
\\
\label{e:PtLam}
\Phi_t =-p-\bodypot + \frac{(\Lambda\Phi)^2-\Phi_\th^2}{2J}
+ \Phi_\th \left( H\left(\frac{\Lambda\Phi}J\right) 
+ \upsilon_1 \right)\,,
\end{gather}
with the definitions
\begin{gather}
J = X_\th^2 + (\Lambda X)^2\,, \qquad 
p = {\gamma}{J^{-3/2}} ( X_\th\,\Lambda X_\th - X_{\th\th}\,\Lambda X) \,.
\label{e:pLam}
\end{gather}

\begin{remark}
We can always require $\Psi=H\Phi$ by choosing $\hat\Psi_0=0$. 
However, $HX=Y-\hat Y_0$, so one needs to know $\hat Y_0=\mb Y(0,t)$
in order to completely reconstruct $Y$ from $X$. 
In case one makes the simple choice $\upsilon_0\equiv0$, 
then because this is equivalent to \eqref{e:th1Z0}, it follows that necessarily
\begin{equation}\label{e:x0ty0t}
\hat X_0'(t) = \hat Y_0'(t) =  0\,,
\end{equation}
whence $\hat X_0(t)$ and $\hat Y_0(t)$ remain constant in time. 
In case one makes the choice \eqref{e:a2}, however,
one should determine $\hat Y_0(t)$ by solving 
the equation 
\begin{equation}\label{e:y0t}
\hat Y_0'(t)  
= -\im \frac{\hat F_1}{\hat Z_1}
= -\Im \frac{\hat\Phi_1}{\hat X_1}\,,
\end{equation}
which derives from \eqref{e:z0t} by applying \eqref{e:ABk} to $\mb F$ and $\mb Z$.
\end{remark}

\subsubsection{Galilean invariance.} One can check (with a bit of difficulty)
that the conformal equations of motion above are invariant under
a Galilean transformation
\begin{equation}
\mb Z = \breve{\mb Z}+{\bm \vel}t\,, 
\quad \mb F = \breve{\mb F}+\bar{\bm\vel}\breve{\mb Z}\,,
\qquad \bm\vel\in\C\,,
\end{equation}
for which
$\mb U = \breve{\mb U}+\bm\vel$, $\mb Z_0 = \breve{\mb Z}_0+\bm\vel t$,
$\mb Z_w=\breve{\mb Z}_w$, and
\[
\mb G = \frac{\mb Z_t}{\mb Z_w} = \breve{\mb G} + \frac{\bm\vel}{{\mb Z}_w}\,.
\]
Defining $\breve\upsilon^n$ and $\breve\upsilon^\tau$ by the analog of 
\eqref{e:barnZt}, we obtain the analog of 
\eqref{e:v0} and \eqref{e:vtau}
due to the fact that the holomorphic function $1/\mb Z_w$ has Taylor series
that starts with the linear approximation
\[
    \frac 1{\mb Z_w(0,t)}- \frac{\mb Z_{ww}(0,t)}{\mb Z_w(0,t)^2}w
    = \hat Z_1\inv - 2\hat Z_2\hat Z_1^{-2} w 
\]
which implies
\[
     \frac{\Im\bm\vel\bar n }J = H\left(\frac{\Re\bm\vel\bar n}J\right) + 
    2\Im \bm\vel(\hat Z_1 e^{-i\th} - \hat Z_2\hat Z_1^{-2})\,.
\]
From \eqref{e:barnZt} and its analog we have 
\[
\upsilon^n = \Re \bar nZ_t = \Re\bar n(\breve Z_t+{\bm\vel}) 
= \breve\upsilon^n
+\Re\bar{\bm\vel}n \,,
\]
while
\[
\Psi_\th = \Im F_\th = \Im(\breve F_\th+\bar{\bm\vel}Z_\th) 
= \breve\Psi_\th
+\Re\bar{\bm\vel}n \,.
\]
Consequently $\breve\upsilon^n=\breve\Psi_\th$ and the relation \eqref{e:match}
transforms as desired.

It is then straightforward to check that the Bernoulli equation
\eqref{e:Ft} transforms properly,
by expanding $\frac12|\breve {\mb U}+\bm\vel|^2$ and adjusting $\tilde c(t)$ to ensure 
that  
\[
    \Re R = \Re \left(\breve R -{\breve{\mb U}}\bar{\bm\vel}- |\bm\vel|^2\right) \,.
\] 

\subsection{Geodesic paths in the conformal mapping formulation}

The governing evolution equations \eqref{e:XtLam}--\eqref{e:PtLam} can be derived from a 
variational principle (principle of stationary action) in a manner following
the derivation in \cite{dksz96}. 
In \cite{dksz96}, Dyachenko et al.~started from Zakharov's canonical Hamiltonian system 
for water waves \cite{zakharov68}, wrote the action as the time integral of a Lagrangian, 
expressed this in terms of conformal variables, and set the first variation to zero.

The essence of this variational derivation is closely related to Arnold's geometric description 
of smooth incompressible Euler flow (in a fixed domain) as geodesic flow in the
group of volume-preserving diffeomorphisms.  Here we show how geodesic conditions
yield the correct equations of motion in the present context of 
ideal droplet flows, governed by two-dimensional incompressible potential flow 
with zero surface tension and zero force potential, in the conformal mapping formulation. Our computation here corresponds to the computation of 
geodesic equations in \cite{LPS} using Eulerian and Lagrangian variables.

Consider a path of conformal injections $\mb Z\colon\bbD\to\C$ that preserve
the area of $\Omega_t=\mb Z(\bbD,t)$, given by 
\begin{equation}\label{e:area0}
|\Omega_t|=\int_{\partial\Omega_t} x\,dy 
= \int_{\bbT} XY_\th\,d\th
= \int_{\bbT} X\,\Lambda X\,d\th
= \frac12 \Im \int_{\bbT} \bar Z Z_\th\,d\th \,.
\end{equation}
We formally define a `Riemannian metric' by pulling back the Eulerian kinetic 
energy of the potential flow induced in $\Omega_t$ by the harmonic
velocity potential $\phi$ whose pulled-back boundary values $\Phi$
are determined (up to constant) by $\mb Z$ and $\mb Z_t$ via \eqref{e:match},
which we can write in the form
\begin{equation}\label{e:match2}
\upsilon^n = \re\bar nZ_t = \Im  Z_\th \bar Z_t = \Psi_\th = \Lambda\Phi \,.
\end{equation}
We note that for the geodesic variational principle,
this kinematic relation \eqref{e:match2} is \textit{imposed} in order to define
a notion of path length corresponding to a suitable `metric.'
This entails the kinematic equation \eqref{e:XtLam} by following the
computations in subsection~\ref{sss:Hkinematic} same as before.

In terms of the holomorphic velocity potential $\phi+i\psi$ on $\Omega_t$
and its pullback $\mb \Phi+i\mb \Psi$ on $\bbD$, 
the kinetic energy is half of the following quantity (where $\nu$
corresponds to the unit outer normal to $\partial\Omega_t$):
\begin{equation}
\int_{\Omega_t}|\nabla\phi|^2 
= \int_{\partial\Omega_t}\phi\, \partial_\nu\phi \,ds
= \int_{\partial\Omega_t}\phi\, \partial_s\psi \,ds
=\int_{\bbT} \Phi\,\Psi_\th\,d\th  .
\end{equation}
We regard this as a quadratic form on the `tangent space' of the `manifold' of
conformal images of $\bbD$, producing 
a `Riemannian metric' at base point $\mb Z$ given by
\begin{equation}\label{e:RMZ}
g_{\mb Z}({\mb Z_t},{\mb Z_t}) =  \int_{\bbT} \Phi\,\Lambda\Phi\,d\th\,.
\end{equation}
This `metric' is naturally degenerate in $\mb Z_t$ 
because the conformal map $\mb Z$
has the same freedom as before: A different choice of $\mb Z_0'$ and $\th_1'(t)$ 
corresponds to a path $\mb Z$ that differs only by 
right composition with a holomorphic automorphism of $\bbD$, and 
the image of $\mb Z$, the Eulerian flow, and the metric do not change.

A geodesic path with respect to this metric is one for which the action
\[
{\mathcal A} = \int_0^T \int_{\bbT} \frac12  \Phi\,\Lambda\Phi\,d\th\,dt
\]
is stationary.  The variation is
\[
\delta{\mathcal A} = \int_0^T \int_{\bbT}  
\Phi\,\Lambda\delta\Phi\,d\th\,dt.
\]
Taking the variation of the kinematic equation \eqref{e:match2}, 
we find
\[
\Lambda\delta\Phi = 
\Im (Z_\th \,\overline{\delta Z}_t - Z_t \,\overline{\delta Z}_\th ) \,.
\]
Substituting in and integrating by parts (requiring $\delta Z=0$
at the endpoints $t=0$ and $t=T$), we get
\begin{equation}\label{e:varA2}
\delta{\mathcal A} = -\Im \int_0^T \int_{\bbT}  
\overline{\delta Z} 
(\Phi_t Z_\th - \Phi_\th Z_t)\,d\th\,dt .
\end{equation}
The variation $\delta Z$ is an arbitrary holomorphic function subject to 
a constraint arising from fixing the area given by \eqref{e:area0}.
From \eqref{e:area0}
we find the constraint on $\delta Z$ takes the form
\begin{equation}\label{c:delZ}
0 = \int_{\bbT}\delta X\,\Lambda X\,d\th = \frac12
\Im \int_{\partial\bbD} \overline{\delta Z}Z_\th \,d\th\,.
\end{equation}

Because the negative Fourier modes of $\delta Z$ vanish, and the non-negative Fourier modes
are arbitrary up to the orthogonality condition \eqref{c:delZ}, we conclude that the factor
$\Phi_t Z_\th - \Phi_\th Z_t$ in the integrand in \eqref{e:varA2} 
is anti-holomophic with mean zero (having only negative Fourier modes), 
modulo a real constant multiple of $Z_\th$.
Thus for some real function $c(t)$,
\begin{equation}\label{e:varA3}
\Phi_t Z_\th - \Phi_\th Z_t = \bar Q  + c(t)Z_\th 
\end{equation}
where $Q:\bbT\to\C$ 
which is the trace of some mean-zero holomorphic function in $\bbD$.
(To see that $c(t)$ must be real, use \eqref{e:varA3} in \eqref{e:varA2}
with $\delta Z=Z_\th$.)

We isolate the factor $\Phi_t$ from this equation as follows.
First, multiply by $\bar Z_\th$, noting that $Z_\th Q$ is
the boundary trace of the holomorphic function $iw\mb Z_w \mb Q$
and this has zero average over $S^1$. 
Next we apply $1+iH$ to both sides, which eliminates the 
term $\overline{Z_\th Q}$ and yields
\[
(1+iH)(J\Phi_t) = (1+iH)(\Phi_\th \bar Z_\th Z_t) + c(t)(1+iH)J
\]
From \eqref{e:barnZt} and \eqref{e:vtau}--\eqref{e:vtau1} we can write
\begin{align}\label{e:Ztth}
\bar Z_\th Z_t &= \bar Z_\th(-iZ_\th)\left(
(1+iH)\frac{\Psi_\th}J + i \upsilon_1 \right)
    = J \left( H\frac{\Psi_\th}J + \upsilon_1 \right) 
 - i \Psi_\th\,.
\nonumber
\end{align}
Substituting this in above and taking the real part we find 
\begin{equation}\label{e:JPhit}
J\Phi_t = J\Phi_\th
    \left( H\frac{\Psi_\th}J + \upsilon_1 \right)
+ H(\Phi_\th\Psi_\th) + c(t)J.
\end{equation}
Now, because $\Psi_\th=H\Phi_\th$, the function
$\Phi_\th+i\Psi_\th$ is a holomorphic trace with mean zero.
Hence so is the square of this function, and consequently
\[
\Phi_\th^2-\Psi_\th^2 = -H(2\Phi_\th\Psi_\th) \,.
\]
Using this in \eqref{e:JPhit} one obtains the Bernoulli equation
\eqref{e:PtLam}
with zero surface tension $\gamma$ and zero force potential $\bodypot$.

\subsubsection{Surface tension and potential energy} 
To include the effects of body force and surface tension 
in the action principle, we subtract the bulk potential energy
and interfacial energy (the droplet perimeter times $\gamma$ here),
to obtain the modified action
\begin{equation}\label{d:AgammaV}
    {\mathcal A}_{\gamma,\bodypot} = 
    \int_0^T 
    \left(\int_{\bbT} \left(\frac12  \Phi\,\Lambda\Phi -\gamma|Z_\th|\right)d\th 
    -\int_{\Omega_t} \bodypot(z)\,dx\,dy\right)\,dt\,, \quad \Omega_t=\mb Z(\bbD,t)\,.
\end{equation}

To determine the variation of bulk potential energy we proceed as follows.
Because a virtual displacement $\delta Z$ corresponds to virtual normal velocity
\[
\tilde v_\nu :=
\Re \overline{\delta Z}\left(\frac{-i Z_\th}{|Z_\th|}\right) 
= \Im \overline{\delta Z}\frac{Z_\th}{|Z_\th|},
\]
using the Reynolds transport theorem we compute
\[
\delta \int_{\Omega_t} \bodypot(z)\,dx\,dy 
= \int_{\partial\Omega_t} \bodypot(z)\tilde v_\nu\,ds
= \Im \int_{\bbT} \overline{\delta Z} \,\bodypot(Z)Z_\th\,d\th.
\]
The variation of droplet perimeter is given by
\[
\delta \int_{\partial\Omega_t} 1\,ds = \delta \int_{\bbT}|Z_\th|\,d\th
= \Re \int_{\bbT}\overline{\delta Z}_\th \frac{Z_\th}{|Z_\th|}\,d\th
= -\Im i\int_{\bbT}\overline{\delta Z} \left(\frac{Z_\th}{|Z_\th|}\right)_\th\,d\th.
\]
Because $Z_\th/|Z_\th|$ is a unit tangent vector to $\partial\Omega_t$
and $ds=|Z_\th|\,d\th$, we have $i(Z_\th/|Z_\th|)_\th=- \kappa Z_\th$ where
$\kappa$ is the curvature of $\partial\Omega_t$.

By consequence, we find
\begin{equation}\label{e:varA2g}
    \delta{\mathcal A}_{\gamma,\bodypot} = -\Im \int_0^T \int_{\bbT}  
\overline{\delta Z} 
(\Phi_t Z_\th - \Phi_\th Z_t+(\gamma\kappa+\bodypot(Z)) Z_\th)\,d\th\,dt ,
\end{equation}
Arguing as before, if this vanishes for all holomorphic $\delta Z$ constrained by 
\eqref{c:delZ}, then we obtain \eqref{e:varA3} with the term $(\gamma\kappa+\bodypot)Z_\th$
added on the left-hand side:
\begin{equation}\label{e:varA4}
\Phi_t Z_\th - \Phi_\th Z_t + (\gamma\kappa+\bodypot)Z_\th = \bar Q  + c(t)Z_\th  \,.
\end{equation}
We then get \eqref{e:JPhit} with the added term $(\gamma\kappa+\bodypot)J$,
and \eqref{e:PtLam} follows as desired.

\begin{remark} 
The fact that the potential energy terms in \eqref{d:AgammaV}
appear with a negative sign is natural due to the origin of this
action from Lagrangian mechanics.
\end{remark}

\subsection{Formal conservation laws for area and energy}
The equations of motion imply conservation of the area $|\Omega_t|$,
because from \eqref{e:area0} and \eqref{e:match2} it follows
\begin{equation}
\frac{d}{dt}|\Omega_t| = \frac12 \int_{\bbT}\Psi_\th\,d\th = 0.
\end{equation}

The variational calculations in the previous subsection also make it convenient
to establish conservation of the total (kinetic plus potential) energy given by 
\begin{equation}\label{d:energy}
{\mathcal E} = \int_{\bbT} \left(\frac12 \Phi\,\Lambda\Phi + \gamma|Z_\th|\right)\,d\th
+ \int_{\Omega_t} \bodypot(z)\,dx\,dy \,.
\end{equation} 
For, taking $\delta Z= Z_t$ in the potential energy computations we find
\begin{equation}
\partial_t\left( \int_{\bbT} \gamma|Z_\th|\,d\th 
+ \int_{\Omega_t} \bodypot(z)\,dx\,dy \right)
= 
\Im \int_{\bbT} \bar Z_t Z_\th (\gamma\kappa + \bodypot) \,d\th
\end{equation}
while, due to \eqref{e:match2},
\begin{equation}
\partial_t \int_{\bbT} \frac12\Phi\,\Lambda\Phi\,d\th = 
\Im\int_{\bbT}
\Phi_tZ_\th \bar Z_t\,d\th \,.
\end{equation}
Adding these equations and using \eqref{e:varA4} we get
\[
\frac{d\mathcal E}{dt} = \Im \int_{\bbT}
(\Phi_\th Z_t +\bar Q + c(t)Z_\th) \bar Z_t\,d\th = 0\,,
\]
because the first term is real, the second term is anti-holomorphic with mean zero, and
in the third term $c(t)$ is real and $\Im Z_\th\bar Z_t=\Psi_\th$, which
integrates to zero.

\subsection{Criteria for continuation of conformal injectivity}
\label{ss:criteria}
We have derived the evolution equations in \eqref{e:XtLam}--\eqref{e:PtLam}
under the assumption that $\mb Z(\cdot,t)$ remains a conformal injection
on the interval of time being considered.
One would then like to have a criterion
which guarantees that having a regular solution of the evolution equations
on this interval of time ensures that the maps $\mb Z(\cdot,t)$
constructed by holomorphic extension from $\partial\bbD$ 
in fact remain conformal injections in $\bbD$. 

Of course we should assume that $\mb Z$ is a conformal injection 
at the initial time $t=0$.  
Suppose $Z=X+iY$ is $C^1$ on $\bbT\times[0,T]$ with $HX=Y-\hat Y_0$,
and let $\mb Z$ be the holomorphic extension of $Z$ into $\bbD$.
If one assumes $J=|Z_\th|>0$ on $\bbT\times[0,T]$, say, then
it follows from the maximum modulus principle and a continuation 
argument that  $1/\mb Z_w$ remains holomorphic and bounded 
on $\bbD\times[0,T]$. Then $\bm Z_w$ will not vanish in $\bbD$,
which is necessary for the map $w\mapsto \mb Z(w,t)$ to be a
conformal injection.

This local condition is in principle not sufficient to ensure
the global injectivity of $\mb Z$ on $\bbD$. 
Non-injectivity of $\mb Z(\cdot,t)$ may develop
due to collision of points on the boundary, 
as in a so-called `splash singularity' (see \cite{CF12}).
We can state a precise sufficient criterion for the continuation of injectivity
as follows, making use of known results in complex function theory.

\begin{theorem}[Persistence of injectivity] 
Let $\mb Z\colon\bar \bbD\times[0,T]\to\C$, 
with $\mb Z(\cdot,t)$ holomorphic in $\bbD$ for each $t$,
and assume $\mb Z$ and $\mb Z_w$ are continuous on $\bar\bbD\times[0,T]$.
Further, assume that $\mb Z(\cdot,0)$ is injective on $\partial\bbD$.
Define 
\[
T_*=\sup\{T_1: \mb Z(\cdot,t)
\mbox{ is injective on $\bar\bbD$ for all $t\in[0,T_1]$} \}.
\]
Then if $T_*<T$, 
the curve $\mb Z(\partial\bbD,T_*)$ fails to be a Jordan curve of class
$C^{1,\alpha}$ for every $\alpha\in (0, 1]$.
\end{theorem}

The point of this result is that persistence of injectivity in $\bbD$
is ensured by good behavior of the boundary trace $Z$.
This is rather subtle since injectivity does not in principle require
nonvanishing of $\bm Z_w$ on the boundary.

 
\begin{proof}
1. Due to the classical Darboux-Picard 
theorem \cite[Thm.~9.16, p.~310]{Burckel79},
the injectivity of $\mb Z(\cdot,t)$ on $\partial\bbD$ 
implies the injectivity of 
$\mb Z(\cdot,t)$ on $\bar\bbD$, for any $t\in[0,T]$.
By consequence $T_*\ge0$. 
Suppose that $T_*<T$. If $\mb Z(\cdot,T_*)$ is
not injective on $\partial\bbD$ then it is not a Jordan curve 
and we are done, so assume it is injective.
Then as above $\mb Z(\cdot,T_*)$ is a conformal injection on $\bar\bbD$.

2.  Now suppose for the sake of contradiction that  $\mb Z(\partial\bbD,T_*)$
is a Jordan curve of class $C^{1,\alpha}$ for some $\alpha\in(0,1)$.
Then, as remarked by Pascu \& Pascu \cite[Remark 2.2]{Pascu11},
results in the book of Pommerenke \cite[p.~24 and pp.~48-49]{Pommerenke}
imply that 
the derivative $\mb Z_w(\cdot,T_*)$ is nonvanishing on $\bar\bbD$.
By consequence, the quantity defined by
\begin{equation}\label{d:K}
K = \inf_{\substack{w_1,w_2\in \bar\bbD \\ w_1\ne w_2}} \left|
\frac{ \mb Z(w_1,T_*)-\mb Z(w_2,T_*)}{w_1-w_2} \right| \,,
\end{equation}
is strictly positive, satisfying $K>0$. 

3. Now, for any function $\mb Y$ holomorphic on $\bbD$ and continuous 
on $\bar\bbD$ that satisfies
\[
\hat K = \sup_{w\in\bbD} |\mb Y_w(w)-\mb Z_w(w,T_*)| < K\,,
\]
necessarily $\mb Y$ is injective on $\bar\bbD$. For if $\mb Y(w_0)=\mb Y(w_1)$
for some $w_0\ne w_1$ in $\bar\bbD$, then with $w_t=(1-t)w_0+tw_1$
we have
\[
|\mb Z(w_1,T_*) -\mb Z(w_0,T_*)| \le 
\int_0^1 |\mb Z_w(w_t,T_*)-\mb Y_w(w_t)|\,dt\,|w_1-w_0| \le \hat K|w_1-w_0|
\]
and this leads to a contradiction with the definition of $K$.
(A more subtle related argument is made in the proof of
Theorem 2.4 of \cite{Pascu11}.)

4. Our assumptions ensure that we may take $\mb Y=\mb Z(\cdot,t)$
whenever $t-T_*>0$ is sufficiently small. This contradicts
the definition of $T_*$ and concludes the proof.
\end{proof}

To relate this criterion back to the solution $(X,\Phi)$ of 
\eqref{e:XtLam}--\eqref{e:PtLam}, we note the following corollary
which is implied by the fact that the Hilbert transform is bounded
on spaces $C^\alpha(\bbT)$ of H\"older continuous functions,
and if the map $\th\mapsto Z(\th,t)$ is $C^{1,\alpha}$ 
{\em with nonvanishing derivative} then the curve $Z(\bbT,t)$ is $C^{1,\alpha}$.

\begin{corollary}\label{c:Xinj}
Suppose $X\in C([0,T],C^{1,\alpha}(\bbT))$ for some $\alpha\in(0,1)$ 
and $\mb Z$ is the holomorphic extension of $Z=X+iHX$ into $\bbD$. 
Then $\mb Z\in C([0,T],C^{1,\alpha}(\bar\bbD))$.
Assume moreover that $Z(\cdot,0)$ is injective on $\bbT$ and define
$T_*$ as in the previous Theorem. Then if $T_*<T$, either
$Z(\cdot,T_*)$ is not injective on $\bbT$, or 
$|Z_\th(\th_*,T_*)|=0$ for some $\th_*\in\bbT$. 
\end{corollary}
\begin{remark} We note that 
in recent work of Kinsey and Wu~\cite{KinseyWu2018} and Wu~\cite{wu2015blow},
certain types of free-surface corners present 
in the initial data have been shown to persist for short time
in solutions of water wave equations.
The results above would say nothing about 
conformal parametriztions of such solutions, since the interface
would fail to be $C^{1,\alpha}$ at all times.
\end{remark}

\bibliographystyle{siam}
\bibliography{eulerrefs,NumAnaPde}
\end{document}